\documentclass[a4paper,12pt]{article}
\usepackage{amssymb,amsmath,amsthm,amsfonts}
\usepackage{bookmark}
\usepackage{mathrsfs}
\usepackage{multirow}
\usepackage{graphicx}
\usepackage{authblk}
\usepackage{indentfirst}
\usepackage{multicol}
\usepackage{tabu}
\usepackage{tabularx}
\usepackage{url}
\usepackage{fancyhdr}
\usepackage[numbers]{natbib}
\usepackage[all]{xy}
\usepackage{mathtools}
\usepackage[english]{babel}
\usepackage{colortbl}
\usepackage{caption}
\usepackage{hyperref}

\newtheorem{theorem}{Theorem}[section]
\newtheorem{proposition}[theorem]{Proposition}
\newtheorem{lemma}[theorem]{Lemma}
\newtheorem{corollary}[theorem]{Corollary}

\newtheorem{definition}{Definition}[section]
\newtheorem{example}{Example}[section]
\newtheorem{remark}{Remark}[section]
\newcommand{\im}{{\mathrm{Im}\hspace{0.1em}}}
\newcommand{\vol}{{\mathbf{vol}\hspace{0.1em}}}

\usepackage{xr}
\makeatletter
    \newcommand*{\addFileDependency}[1]{
    \typeout{(#1)}
    \@addtofilelist{#1}
    \IfFileExists{#1}{}{\typeout{No file #1.}}
    }
\makeatother

\title{Persistent  hyperdigraph homology and persistent hyperdigraph Laplacians}

\author[1]{Dong Chen }
\author[2,3]{Jian Liu \footnote{The first two authors contribute equally.}}
\author[3]{Jie Wu \thanks{Corresponding author: wujie@bimsa.cn}}
\author[1,4,5]{Guo-Wei Wei \thanks{Corresponding author: weig@msu.edu}}
\affil[1]{Department of Mathematics, Michigan State University, MI, 48824, USA}
\affil[2]{Mathematical Science Research Center, Chongqing University of Technology, Chongqing 400054, China}
\affil[3]{Yanqi Lake Beijing Institute of Mathematical Sciences and Applications, Beijing 101408, China}
\affil[4]{Department of Electrical and Computer Engineering, Michigan State University, MI 48824, USA}
\affil[5]{Department of Biochemistry and Molecular Biology, Michigan State University, MI 48824, USA}

\makeatletter
    \renewcommand*{\@fnsymbol}[1]{\ensuremath{\ifcase#1\or \dagger\or *\or *\or
   \mathsection\or \else\@ctrerr\fi}}
\makeatother
\date{}

\begin{document}
    % \linenumbers
    \maketitle

    \paragraph{Abstract}

Hypergraphs are useful mathematical models for describing complex relationships among members of a structured graph, while hyperdigraphs serve as a generalization that can encode asymmetric relationships in the data. However, obtaining topological information directly from hyperdigraphs remains a challenge.
To address this issue, we introduce hyperdigraph homology in this work. We also propose topological hyperdigraph Laplacians, which can extract both harmonic spectra and non-harmonic spectra from directed and internally organized data. Moreover, we introduce persistent hyperdigraph homology and persistent hyperdigraph Laplacians through filtration, enabling the capture of topological persistence and homotopic shape evolution of directed and structured data across multiple scales. The proposed methods offer new multiscale  algebraic topology tools for topological data analysis.


    \newpage
    \tableofcontents
    \newpage

\section{Introduction}\label{section:introduction}

Topology and homology study the invariant properties of geometric objects under continuous deformations, providing a high level of abstraction for these objects \cite{kaczynski2004computational}. The well-known joke that topologists cannot distinguish between a coffee mug and a doughnut highlights the difficulty of topology in describing real-world objects. However, topological data analysis (TDA) has recently emerged as a way to overcome this difficulty. TDA facilitates topological deep learning, an emerging paradigm in data science that has been successful in various applications \cite{townsend2020representation,cang2017topologynet}. The main tool of TDA is persistent homology, which creates a family of multiscale topological spaces from a given dataset by filtration, allowing the extraction and analysis of the topological invariants of the data at various scales \cite{zomorodian2004computing,edelsbrunner2008persistent, bubenik2017persistence}. Through comparative analysis, persistent homology can be used to infer the shape of data \cite{carlsson2009topology}. However, a limitation of persistent homology is that at each dimension, the Betti number only counts the number of independent components and does not describe the properties of each component. For instance, a heterogeneous ring is counted the same as a homogeneous ring at dimension 1. To address this limitation, persistent cohomology was introduced, which embeds both geometric and non-geometric properties of the data into topological invariants \cite{cang2020persistent}.
Moreover, topological invariants are qualitative rather than quantitative. For example, at dimension 1, a ring with five members is counted the same as a ring with six members. These issues can limit the power of persistent homology in network analysis and other applications.

The graph Laplacian was originally introduced by Kirchhoff in 1847 to analyze electrical networks \cite{kirchhoff1847ueber}. For instance, the  second-smallest eigenvalue, also known as the Fiedler eigenvalue \cite{fiedler1973algebraic}, describes the algebraic connectivity of a graph. In 1944, Eckmann generalized the graph Laplacian to the simplicial complex setting, resulting in what is now known as combinatorial Laplacians. Combinatorial Laplacians are the discrete counterpart of Hodge Laplacians, which are also called Laplace-de Rham operators in differential geometry. The associated de Rham cohomology is often used to study the topology of manifolds and the behavior of vector fields on them. The connection between Hodge Laplacians on a differentiable manifold and combinatorial Laplacians on a point cloud is non-trivial and was studied in the context of discrete exterior calculus to understand discrete equivalents of differential forms \cite{hirani2003discrete,desbrun2005discrete,arnold2006finite}.
A notable property of both Hodge Laplacians and combinatorial Laplacians is that their harmonic spectra give rise to corresponding topological invariants or Betti numbers, which is why we refer to them as topological Laplacians \cite{wei2023topological}. There are several other topological Laplacians, such as sheaf Laplacians defined on cellular sheaves \cite{hansen2019toward} and path Laplacians  \cite{gomes2019path} derived from path complexes and path homology introduced by Yau and coworkers \cite{grigor2012homologies,grigor2020path}. Path homology and persistent path homology provide a topological analysis of digraphs and offer promising applications in molecular and material sciences \cite{chen2023path}.
Compared to corresponding homology theories, topological Laplacians are capable of describing the quantitative properties of the underlying system in their non-harmonic spectra \cite{horak2013spectra}.

In 2019, Wei and coworkers extended the power of topological Laplacians by introducing persistent topological Laplacians, which offer superior performance over persistent homology \cite{chen2019evolutionary,wang2019persistent}. Persistent Laplacian, also known as persistent spectral graphs or persistent combinatorial Laplacian, was introduced using  filtration by Wang et al.  \cite{wang2019persistent}. It  has been studied mathematically by Memoli et al. \cite{memoli2022persistent} and Liu et al.  \cite{liu2023algebraic}  and explored in biological contexts by Meng et al. \cite{meng2021persistent}, Chen et al. \cite{chen2022persistent}, Wee et al. \cite{wee2022persistent}, and Qiu and Wei \cite{qiu2023persistent}. An open-source online package has also been developed \cite{wang2021hermes}.
The evolutionary de Rham-Hodge Laplacians, or persistent Hodge Laplacians, were defined on a family of evolving manifolds \cite{chen2021evolutionary}. Discrete exterior calculus was utilized to implement various boundary conditions on manifolds with boundaries, which helped to accurately compute topological invariants from the harmonic spectra of persistent Hodge Laplacian operators. Persistent Hodge Laplacians provided a mathematical model for musical instruments that cannot be described by persistent homology \cite{wei2023topological}. Additionally, Wei and coworkers introduced persistent sheaf Laplacians on cellular sheaves \cite{wei2021persistent} and persistent path Laplacians on path complexes \cite{wang2023persistent}. The former allows for the embedding of heterogeneous information, such as atomic partial charges, into topological variants, while the latter facilitates the topological analysis of digraphs and directed networks. Ameneyro et al. proposed persistent Dirac operators for the efficient quantum computation of persistent Betti numbers \cite{ameneyro2022quantum}. This approach utilizes the relationship between Dirac operators and Laplacian operators.
 It has been demonstrated that persistent (topological) Laplacians can capture the homotopic shape evolution of data that is not present in the corresponding persistent homology analysis \cite{wei2023topological}.  An interesting property of non-harmonic eigenvalues is that they are discontinuous when there is a topological change.
 These new persistent topological  Laplacians have significantly expanded the application domain and power of TDA.

None of the methods mentioned earlier describe the internal organization of a network. Hypergraph is a popular mathematical model for data with complex relationships, and it has been widely applied in physics \cite{qu2013encoding}, computer science \cite{eiter2002hypergraph}, and engineering \cite{akhremtsev2017engineering}. However, traditional hypergraphs do not capture the topological information in the data. To address this issue, Wu and coworkers generalized traditional hypergraphs into topological hypergraphs using simplicial complexes \cite{ren2018hodge,bressan2019embedded}. The embedded homology of hypergraphs was introduced to study the topological invariants of hypergraphs \cite{bressan2019embedded}. Hodge-decomposition type of weighted hypergraphs was also studied \cite{ren2018hodge}, and discrete Morse functions for hypergraphs were considered \cite{ren2018discrete}. Additionally, Grbic et al. introduced the concept of super-hypergraphs and studied the embedded homology of super-hypergraphs \cite{grbic2022aspects}. More recently, Liu et al. introduced filtration to topological hypergraph Laplacians   \cite{liu2021persistent}.

Hypergraphs do not apply to   directed graphs (digraphs) and directed networks. To address this limitation, hyperdigraphs have been introduced as a generalization of digraphs \cite{gallo1993directed,berge1973graphs,dorfler1980category}. Essentially, a hyperdigraph is a hypergraph with an additional direction assigned to each hyperedge. As stated in \cite{ausiello2017directed}, ``hyperdigraphs have been applied in several domains where we are interested in representing implication systems (database theory, logics, artificial intelligence, etc.)". As a powerful mathematical model for data analysis, hyperdigraphs have been widely applied in various fields of computer science \cite{gallo1998directed,ramaswamy1997using}. Since hyperdigraphs can encode the orientation information between different objects, they have a natural advantage for material structure modeling, molecular group interaction analysis, biological system analysis, and more.
However, because of the complexity of hyperdigraphs, there is currently no established framework for topological hyperdigraphs or  hyperdigraph homology. Efforts have been made to analyze hyperdigraphs by identifying path complexes on them with the help of path homology theory \cite{muranov2021path}. Nevertheless, constructing topological hyperdigraphs requires the use of embedded homology techniques \cite{ren2018hodge,bressan2019embedded} or equivalent methods specifically tailored for hyperdigraph homology.

This work introduces several new TDA models: hyperdigraph homology (HDGH), topological hyperdigraph Laplacians (THDGLs), persistent hyperdigraph homology (PHDGH), and persistent hyperdigraph Laplacians (PHDGLs). To introduce a topological structure, we use sequences as the building blocks for hyperdigraph homology. The chain complex of the hyperdigraph homology is the maximal chain complex embedded into the space generated by these sequences. We develop an embedded homology technique to create HDGH.
To construct topological hyperdigraph Laplacians, we use boundary operators and adjoints. Additionally, we define PHDGH and PHDGLs by equipping a filtration process to analyze geometric objects at various filtration scales. We propose both a volume-based filtration and a distance-based filtration for PHDGH and PHDGLs.

The remainder of the paper is organized as follows. In the next section, we review hypergraph homology and topological hypergraph Laplacians to establish notations. In Section \ref{section:hyperdigraph}, we propose hyperdigraph homology and topological hyperdigraph Laplacians. In Section \ref{section:persistence_on_hyper}, we define persistence on hyperdigraph homology and topological hyperdigraph Laplacians. Finally,  we demonstrate the application of our proposed persistent hyperdigraph Laplacians to a protein-ligand complex in Section \ref{section:application}.

\section{A brief review of topological hypergraphs}\label{section:hypergraph}

 Hypergraph, as a kind of generalization of the simplicial complex, has attracted considerable attention in theory and application. The embedded homology provides a realization of topological hypergraphs \cite{bressan2019embedded}. In this section, we  review   hypergraph homology and  topological hypergraph Laplacians. From now on, $\mathbb{K}$ is always assumed to be a ground field and $|X|$ denotes the number of elements in a finite set $X$.

  Let $V$ be a nonempty finite ordered set. Let $\mathbf{P}(V)$ denote the power set of $V$ excluding the empty set, that is, the set of nonempty subsets of $V$. So we have that $\mathbf{P}(V)=\coprod\limits_{n=1}^{|V|}\mathbf{P}_{n}(V)$, where $\mathbf{P}_{n}(V)$ is the set of subsets with $n$ elements in $V$.
  A \emph{hypergraph} is a pair $\mathcal{H}=(V,E)$, where $E$ is a subset of $\mathbf{P}(V)$. The set $E$ is called the hyperedge set and an element $e\in E\cap \mathbf{P}_{n+1}(V)$ is called an $n$-hyperedge.
  In particular, the hypergraph $\mathcal{H}(V)=(V,\mathbf{P}(V))$ is called a \emph{complete hypergraph}. A hypergraph $\mathcal{H}=(V,E)$ can be regarded as a sub hypergraph of its associated complete hypergraph $\mathcal{H}(V)$. It is worth noting that a complete hypergraph is exactly an abstract simplicial complex.

  Now, let $\mathcal{H}=(V,E)$ be a hypergraph. Let $C_{p}(\mathcal{H};\mathbb{K})$ be the $\mathbb{K}$-linear space generated by the $p$-hyperedges of $\mathcal{H}$. Then $C_{\ast}(\mathcal{H};\mathbb{K})=(C_{p}(\mathcal{H};\mathbb{K}))_{p\geq 0}$ is a graded $\mathbb{K}$-linear space. In particular, let $C_{p}(V;\mathbb{K})$ be the $\mathbb{K}$-linear space generated by the $p$-hyperedges of $\mathcal{H}(V)$. Then $C_{\ast}(V;\mathbb{K})=(C_{p}(V;\mathbb{K}))_{p\geq 0}$ is a chain complex with the boundary operator $d_{p}:C_{p}(V;\mathbb{K})\to C_{p-1}(V;\mathbb{K})$ given by
  \begin{equation}
    d_{p}\{v_{0},v_{1},\dots,v_{p}\}=\sum\limits_{i=0}^{p}(-1)^{i}\{v_{0},\dots,\widehat{v_{i}},\dots,v_{p}\},\quad p\geq 1.
  \end{equation}
  Here, $\widehat{v_{i}}$ means the element $v_{i}$ is omitted. We make the convention $d_{0}\{v_{0}\}=0$ for all $v_{0}\in V$.
  The boundary operator $d_{p}$ on $C_{p}(V;\mathbb{K})$ restricts to a map $d_{p}:C_{p}(\mathcal{H};\mathbb{K})\to C_{p-1}(V;\mathbb{K})$. Then we have the infimum chain complex defined by
  \begin{equation}\label{equation:infimum_chaincomplex}
    \mathrm{Inf}_{p}(\mathcal{H};\mathbb{K})=\{x\in C_{p}(\mathcal{H};\mathbb{K})|d_{p}x\in C_{p-1}(\mathcal{H};\mathbb{K})\}
  \end{equation}
  for $p\geq 1$ and $\mathrm{Inf}_{0}(\mathcal{H};\mathbb{K})=C_{0}(\mathcal{H};\mathbb{K})$. It can be verified that $\mathrm{Inf}_{\ast}(\mathcal{H};\mathbb{K})=(\mathrm{Inf}_{p}(\mathcal{H};\mathbb{K}))_{p\geq 0}$ is indeed a chain complex.
  \begin{definition}
  The \emph{embedded homology} of a hypergraph $\mathcal{H}$ is defined by
  \begin{equation}
    H_{p}(\mathcal{H};\mathbb{K})=H_{p}(\mathrm{Inf}_{\ast}(\mathcal{H};\mathbb{K})),\quad p\geq 0.
  \end{equation}
  \end{definition}
  In this work, the topological hypergraph or hypergraph homology considered is the embedded homology of hypergraphs. The corresponding \emph{$p$-th Betti number} for the hypergraph $\mathcal{H}$ is defined by $\beta_{p}=\dim H_{p}(\mathcal{H};\mathbb{K})$ for $p\geq 0$.
  \begin{proposition}[\cite{bressan2019embedded}]
  The homology $H_{\ast}(-;\mathbb{K}):\mathbf{Hyper}\to \mathbf{Vec}_{\mathbb{K}}$ is a functor from the category of hypergraphs to the category of $\mathbb{K}$-linear spaces.
  \end{proposition}
  Now, consider the case $\mathbb{K}=\mathbb{R}$. Endow $C_{\ast}(\mathcal{H};\mathbb{K})$ with the standard inner product given by
  \begin{equation*}
    \langle x,y\rangle=\left\{
                         \begin{array}{ll}
                           1, & \hbox{$x=y$;} \\
                           0, & \hbox{otherwise.}
                         \end{array}
                       \right.
  \end{equation*}
  Here, $x,y$ are hyperedges of $\mathcal{H}$. Thereby, $C_{\ast}(\mathcal{H};\mathbb{R})$ is an finite-dimensional inner product space. The infimum chain complex $\mathrm{Inf}_{\ast}(\mathcal{H};\mathbb{R})$ inherits the inner product structure of $C_{\ast}(\mathcal{H};\mathbb{R})$. We also denote the boundary operator on $\mathrm{Inf}_{p}(\mathcal{H};\mathbb{R})$ by $d_{p}:\mathrm{Inf}_{p}(\mathcal{H};\mathbb{R})\to \mathrm{Inf}_{p-1}(\mathcal{H};\mathbb{R})$. Then, we have the adjoint operator $(d_{p})^{\ast}:\mathrm{Inf}_{p-1}(\mathcal{H};\mathbb{R})\to \mathrm{Inf}_{p}(\mathcal{H};\mathbb{R})$ of $d_{p}$ with respect to the above inner product.
  \begin{definition}
  The \emph{topological hypergraph Laplacian} $\Delta^{\mathcal{H}}_{p}:\mathrm{Inf}_{p}(\mathcal{H};\mathbb{R})\to \mathrm{Inf}_{p}(\mathcal{H};\mathbb{R})$ of $\mathcal{H}$ is defined by
  \begin{equation*}
    \Delta^{\mathcal{H}}_{p}=(d_{p})^{\ast}\circ d_{p}+d_{p+1}\circ (d_{p+1})^{\ast},\quad p\geq 0.
  \end{equation*}
  \end{definition}
  In particular, if $\mathcal{H}$ is an abstract simplicial complex, then the embedded homology of $\mathcal{H}$ coincides with the simplicial homology of $\mathcal{H}$. Moreover, the topological hypergraph Laplacian of $\mathcal{H}$ reduces to the Laplacian of simplicial complexes.

  We choose a family of standard orthogonal bases of $\mathrm{Inf}_{p}(\mathcal{H};\mathbb{R})$ for $p\geq 0$. Let $B_{p}$ be the representation matrix of $d_{p}$ with respect to the chosen standard basis. Then the representation matrix \footnote{Here, the representation matrix is given by the left multiplication action on the basis. For the right multiplication action on the basis, the Laplacian matrix $L_{p}^{\mathcal{H}}=B_{p+1}B_{p+1}^{T}+B_{p}^{T}B_{p}$.} of $\Delta^{\mathcal{H}}_{p}$ is given by
  \begin{equation}\label{equation:laplacian}
    L_{p}^{\mathcal{H}}=B_{p+1}^{T}B_{p+1}+B_{p}B_{p}^{T}.
  \end{equation}
  Here, $B_{p}^{T}$ denote the transpose matrix of $B_{p}$.
  The following proposition shows that all the eigenvalues of $\Delta^{\mathcal{H}}_{p}$ is non-negative.
  \begin{proposition}
  The operator $\Delta^{\mathcal{H}}_{p}$ is self-adjoint and non-negative definite.
  \end{proposition}
  \begin{proof}
  Obviously, for any $x,y\in\mathrm{Inf}_{p}(\mathcal{H};\mathbb{R})$, we have
  \begin{equation*}
    \langle \Delta^{\mathcal{H}}_{p}x,y\rangle=\langle d_{p}x,d_{p}y\rangle+\langle (d_{p+1})^{\ast}x,(d_{p+1})^{\ast}y\rangle=\langle x,\Delta^{\mathcal{H}}_{p}y\rangle.
  \end{equation*}
  On the other hand, one has
  \begin{equation*}
    \langle \Delta^{\mathcal{H}}_{p}x,x\rangle=\langle d_{p}x,d_{p}x\rangle+\langle (d_{p+1})^{\ast}x,(d_{p+1})^{\ast}x\rangle\geq 0.
  \end{equation*}
  Thus $\Delta^{\mathcal{H}}_{p}$ is self-adjoint and non-negative definite.
  \end{proof}
  By the algebraic Hodge decomposition theorem, one has the following.
  \begin{proposition}
  $\mathrm{Inf}_{p}(\mathcal{H};\mathbb{R})=\ker \Delta^{\mathcal{H}}_{p}\oplus \mathrm{Im}d_{p}\oplus\mathrm{Im}(d_{p+1})^{\ast}$. Moreover, we have $\ker \Delta^{\mathcal{H}}_{p}=\ker d_{p}\cap \ker (d_{p+1})^{\ast}\cong H_{p}(\mathcal{H};\mathbb{R})$.
  \end{proposition}
  The above proposition shows that the number of zero eigenvalues of $\Delta^{\mathcal{H}}_{p}$ equals to the $p$-th Betti number of $\mathcal{H}$.
  \begin{example}\label{example:hypergraph1}
  Let $\mathcal{H}=(V,E)$ be a hypergraph with $V=\{1,2,3\}$ and
  \begin{equation*}
    E=\{\{1\},\{2\},\{1,2\},\{1,3\},\{2,3\},\{1,2,3\}\}.
  \end{equation*}
  Thus, one has
  \begin{eqnarray*}
    C_{0}(\mathcal{H};\mathbb{R}) &=& \mathrm{span}\{\{1\},\{2\}\}, \\
    C_{1}(\mathcal{H};\mathbb{R}) &=& \mathrm{span}\{\{1,2\},\{1,3\},\{2,3\}\}, \\
    C_{2}(\mathcal{H};\mathbb{R}) &=& \mathrm{span}\{\{1,2,3\}\}.
  \end{eqnarray*}
  By Eq. (\ref{equation:infimum_chaincomplex}), we obtain
  \begin{eqnarray*}
    \mathrm{Inf}_{0}(\mathcal{H};\mathbb{R}) &=& \mathrm{span}\{\{1\},\{2\}\}, \\
    \mathrm{Inf}_{1}(\mathcal{H};\mathbb{R}) &=& \mathrm{span}\{\{1,2\},\{1,3\}-\{2,3\}\}, \\
    \mathrm{Inf}_{2}(\mathcal{H};\mathbb{R}) &=& \mathrm{span}\{\{1,2,3\}\}.
  \end{eqnarray*}
  The infimum complex $\mathrm{Inf}_{\ast}(\mathcal{H};\mathbb{R})$ refers to the maximal sub chain complex contained in the $\mathbb{K}$-linear space $C_{\ast}(\mathcal{H};\mathbb{R})$.
  Here, the standard orthogonal basis \footnote{It is to point out that in the work of persistent path Laplacian \cite{wang2023persistent}, the standard orthogonal basis was not chosen, and as a result, the calculated non-zero eigenvalues were not unique.} of $\mathrm{Inf}_{\ast}(\mathcal{H};\mathbb{R})$ is chosen as follows,
  \begin{equation*}
  \{1\},\{2\},\{1,2\},\frac{1}{\sqrt{2}}\{1,3\}-\frac{1}{\sqrt{2}}\{2,3\},\{1,2,3\}.
  \end{equation*}
  Then the boundary matrices are given by
  \begin{eqnarray*}
   d_0\left(
                       \begin{array}{cc}
                        \{1\}\\ \{2\} \\
                       \end{array}
                     \right)&=&\left(
                       \begin{array}{cc}
                        0\\ 0 \\
                       \end{array}
                     \right),\\
    d_{1}\left(
           \begin{array}{c}
             \{1,2\} \\
             \frac{1}{\sqrt{2}}\{1,3\}-\frac{1}{\sqrt{2}}\{2,3\} \\
           \end{array}
         \right) &=& \left(
                   \begin{array}{cc}
                     -1 & 1 \\
                     -\frac{1}{\sqrt{2}} & \frac{1}{\sqrt{2}} \\
                   \end{array}
                 \right)\left(
           \begin{array}{c}
             \{1\} \\
             \{2\} \\
           \end{array}
         \right), \\
     d_{2}\{1,2,3\} &=& \left(
                     \begin{array}{cc}
                       1 & -\sqrt{2} \\
                     \end{array}
                   \right)\left(
           \begin{array}{c}
             \{1,2\} \\
             \frac{1}{\sqrt{2}}\{1,3\}-\frac{1}{\sqrt{2}}\{2,3\} \\
           \end{array}
         \right).
  \end{eqnarray*}
  Thus the representation matrices of the Laplacians $\Delta^{\mathcal{H}}_{0},\Delta^{\mathcal{H}}_{1}, $ and $ \Delta^{\mathcal{H}}_{2}$ are listed as follows.
  \begin{equation*}
    L_{0}^{\mathcal{H}}=\left(
                   \begin{array}{cc}
                     -1 & 1 \\
                     -\frac{1}{\sqrt{2}} & \frac{1}{\sqrt{2}} \\
                   \end{array}
                 \right)^{T}\left(
                   \begin{array}{cc}
                     -1 & 1 \\
                     -\frac{1}{\sqrt{2}} & \frac{1}{\sqrt{2}} \\
                   \end{array}
                 \right)=\left(
                   \begin{array}{cc}
                     \frac{3}{2} & -\frac{3}{2} \\
                     -\frac{3}{2} & \frac{3}{2} \\
                   \end{array}
                 \right).
  \end{equation*}
  \begin{equation*}
   L_{1}^{\mathcal{H}}=\left(
                  \begin{array}{cc}
                    1 & -\sqrt{2} \\
                  \end{array}
                \right)^{T}\left(
                  \begin{array}{cc}
                    1 & -\sqrt{2} \\
                  \end{array}
                \right)+\left(
                   \begin{array}{cc}
                    -1 & 1 \\
                    -\frac{1}{\sqrt{2}} & \frac{1}{\sqrt{2}} \\
                   \end{array}
                 \right)\left(
                   \begin{array}{cc}
                    -1 & 1 \\
                    -\frac{1}{\sqrt{2}} & \frac{1}{\sqrt{2}} \\
                   \end{array}
                 \right)^{T}=\left(
                             \begin{array}{cc}
                               3 & 0 \\
                               0 & 3 \\
                             \end{array}
                           \right).
  \end{equation*}
  \begin{equation*}
    L_{2}^{\mathcal{H}}=\left(
                     \begin{array}{cc}
                       1 & -\sqrt{2} \\
                     \end{array}
                   \right)\left(
                     \begin{array}{cc}
                       1 & -\sqrt{2} \\
                     \end{array}
                   \right)^{T}=3.
  \end{equation*}
 Therefore,  the spectra for the Laplacian matrices $L_{0}^{\mathcal{H}},L_{1}^{\mathcal{H}}, $and$ L_{2}^{\mathcal{H}}$ are $\{0,3\},\{3,3\}, $ and $\{3\}$, respectively.
  \end{example}

  \begin{figure}[!ht]
    \centering
    \includegraphics[width=14cm]{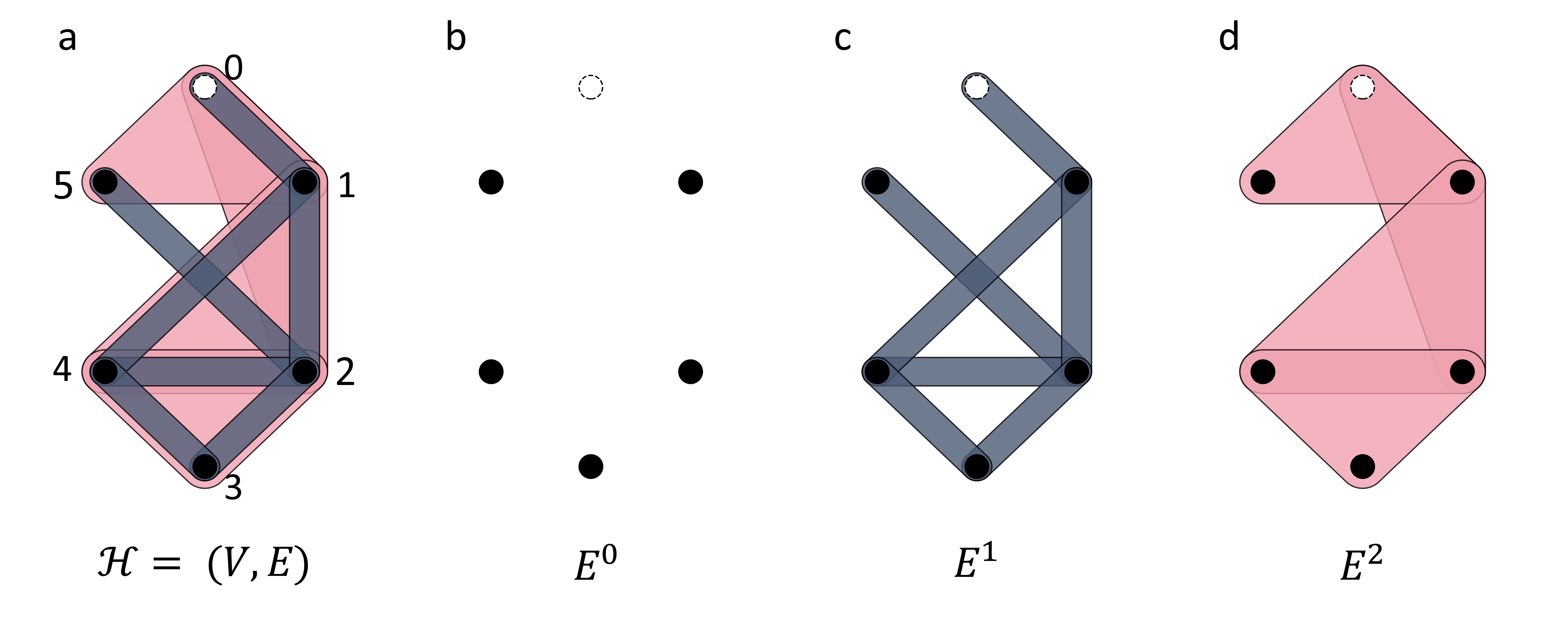}
    % \captionsetup{width=16cm}
    \caption{
      {\bf a} Illustration of hypergraph $\mathcal{H}$ in Example \ref{example:hypergraph2}.
      {\bf b}, {\bf c}, and {\bf d} Illustration of 0-hyperedges, 1-hyperedges, and 2-hyperedges, respectively. The solid black vertices indicate 0-hyperedges, purple edges indicate 1-hyperedges, and pink areas indicate 2-hyperedges. The dashed hollow circle indicates a vertex that is not 0-hyperedge.}
    \label{figure:hypergraph2}
  \end{figure}

  \begin{example}\label{example:hypergraph2}
    As shown in Figure \ref{figure:hypergraph2}{\bf a}, let $\mathcal{H}=(V,E)$ be a hypergraph with vertex set $V=\{0, 1, 2, 3, 4, 5\}$ and the hyperedge set $E$, where $E$ contains the 0-hyperedges $E^0=\{\{1\}, \{2\}, \{3\}, \{4\}, \{5\}\}$ (Figure \ref{figure:hypergraph2}{\bf b}), 1-hyperedges $E^1 = \{\{0, 1\}, \{1, 2\}, \{1, 4\}, \{2, 3\}, \{2, 4\}$, $\{2, 5\}, \{3, 4\}\}$ (Figure \ref{figure:hypergraph2}{\bf c}), and the 2-hyperedges $E^2=\{\{0, 1, 2\}, \{0, 1, 5\}, \{1, 2, 4\}, \{2, 3, 4\}\}$ (Figure \ref{figure:hypergraph2}{\bf d}). To follow Eq. (\ref{equation:infimum_chaincomplex}), the standard orthogonal basis of $\mathrm{Inf}_{*}(\mathcal{H};\mathbb{R})$ is chosen as follows,
%    \begin{eqnarray*}
%      \mathrm{Inf}_{0}(\mathcal{H};\mathbb{R}) &=& \mathrm{span}\{\{1\},\{2\},\{3\},\{4\},\{5\}\}, \\
%      \mathrm{Inf}_{1}(\mathcal{H};\mathbb{R}) &=& \mathrm{span}\{\{1,2\}, \{1,4\}, \{2,3\}, \{2,4\},\{2, 5\}, \{3, 4\}\}, \\
%      \mathrm{Inf}_{2}(\mathcal{H};\mathbb{R}) &=& \mathrm{span}\{\{1, 2, 4\}, \{2, 3, 4\}\}.
%    \end{eqnarray*}
    \begin{equation*}
      \{1\},\{2\},\{3\},\{4\},\{5\};\{1,2\}, \{1,4\}, \{2,3\}, \{2,4\},\{2, 5\}, \{3, 4\};\{1, 2, 4\}, \{2, 3, 4\}.
    \end{equation*}
    Then, the representation matrices of $d_0$,$d_1$, and $d_2$ are given as follows,
    \begin{eqnarray*}
      d_0\left(
                       \begin{array}{cc}
                        \{1\}\\ \{2\}\\ \{3\}\\ \{4\}\\ \{5\} \\
                       \end{array}
                     \right)&=&\left(
                       \begin{array}{cc}
                        0\\ 0\\ 0\\ 0\\ 0 \\
                       \end{array}
                     \right), \\
      d_{1}\left(
          \begin{array}{c}
            \{1,2\}\\ \{1,4\}\\ \{2,3\}\\ \{2,4\}\\ \{2, 5\}\\ \{3, 4\}\\
          \end{array}
        \right)&=&\left(
               \begin{array}{ccccc}
                -1& 1& 0&0&0\\
                -1& 0& 0&1&0\\
                 0&-1& 1&0&0\\
                 0&-1& 0&1&0\\
                 0&-1& 0&0&1\\
                 0& 0&-1&1&0\\
               \end{array}
             \right)\left(
                       \begin{array}{cc}
                        \{1\}\\ \{2\}\\ \{3\}\\ \{4\}\\ \{5\} \\
                       \end{array}
                     \right), \\
         d_{2}\left(
          \begin{array}{c}
            \{1, 2, 4\}\\ \{2, 3, 4\}\\
          \end{array}
         \right) &=& \left(
                         \begin{array}{cccccc}
                          1&-1&0& 1&0&0\\
                          0& 0&1&-1&0&1\\
                         \end{array}
                       \right)\left(
               \begin{array}{c}
                \{1,2\}\\ \{1,4\}\\ \{2,3\}\\ \{2,4\}\\ \{2, 5\}\\ \{3, 4\}\\
               \end{array}
             \right).
    \end{eqnarray*}
    Based on Eq. (\ref{equation:laplacian}), the representation matrices of the Laplacians are listed as follows,
    \begin{eqnarray*}
      L_{0}^{\mathcal{H}}&=&\left(
                     \begin{array}{ccccc}
                      2& -1&  0& -1& 0\\
                      -1&  4& -1& -1&-1\\
                       0& -1&  2& -1& 0\\
                      -1& -1& -1&  3& 0\\
                       0& -1&  0&  0& 1\\
                     \end{array}
                   \right), \\
      L_{1}^{\mathcal{H}}&=&\left(
                    \begin{array}{cccccc}
                      3& 0& -1&  0& -1& 0\\
                      0& 3&  0&  0&  0& 1\\
                     -1& 0&  3&  0&  1& 0\\
                      0& 0&  0&  4&  1& 0\\
                     -1& 0&  1&  1&  2& 0\\
                      0& 1&  0&  0&  0& 3\\
                    \end{array}
                  \right), \\
      L_{2}^{\mathcal{H}}&=&\left(
                    \begin{array}{cc}
                      3& -1\\
                      -1& 3\\
                    \end{array}
                  \right).
    \end{eqnarray*}
  Then, the spectra of Laplacian matrices can be generated, which are $\mathbf{Spec}(L_{0}^{\mathcal{H}})=\{0, 1, 2, 4, 5\}$, $\mathbf{Spec}(L_{1}^{\mathcal{H}})=\{1, 2, 2, 4, 4, 5\}$, and $\mathbf{Spec}(L_{2}^{\mathcal{H}})=\{2,4\}$. The Betti numbers are $\beta_0=1, \beta_1=0,  $ and $ \beta_2=0$, and the smallest eigenvalues of the non-harmonic spectra for $L_{0}^{\mathcal{H}},L_{1}^{\mathcal{H}}, $ and $ L_{2}^{\mathcal{H}}$ are $\lambda_0=1,\lambda_1=1$, and $\lambda_2=2$, respectively.

  \end{example}

  %Regarding a simplicial complex as a hypergraph, the embedded homology of the hypergraph coincides with the simplicial homology of simplcial complexes.
  %The embedded homology is essentially from the idea of the GLMY theory \cite{grigor2012homologies,grigor2017homologies,grigor2020path}. The extended theories for embedded homology are developed such as Hodge Decompositions for weighted hypergraph \cite{ren2018hodge}, discrete Morse theory for hypergraphs\cite{ren2018discrete}, and the super-hypergraphs homology \cite{grbic2022aspects}.

\section{Hyperdigraph homology  and topological hyperdigraph Laplacians}\label{section:hyperdigraph}

%  There are various definitions of the hyperdigraph \cite{ausiello2017directed,thakur2009linear}. Besides, different terms, such as directed hypergraphs and oriented hypergraphs, have appeared to refer to hyperdigraphs. In this paper, we follow the definition of the hyperdigraph in the reference \cite{berge1984hypergraphs} and study the topological homology and Laplacian for hyperdigraphs. The idea is mainly from the GLMY theory \cite{grigor2012homologies,grigor2017homologies,grigor2020path} and the embedded homology of hypergraphs \cite{bressan2019embedded}.

%There are various definitions of the hyperdigraph \cite{ausiello2017directed,thakur2009linear}. Additionally, different terms, such as directed hypergraphs and oriented hypergraphs, have been used to refer to hyperdigraphs. For the purposes of this paper, we adopt the definition of hyperdigraphs as given in \cite{berge1984hypergraphs}, and explore the topological homology and topological Laplacian for hyperdigraphs. Our approach is inspired by the path homology \cite{grigor2012homologies,grigor2017homologies,grigor2020path} and the embedded homology of hypergraphs \cite{bressan2019embedded}.

There are various definitions of   hyperdigraphs \cite{ausiello2017directed,thakur2009linear}. Additionally, both directed hypergraphs and oriented hypergraphs are used to refer to hyperdigraphs. For the sake of simplicity,  we adopt the definition of hyperdigraphs as given in \cite{berge1984hypergraphs}, in which the hyperdigraph consists of sequences of distinct elements in a finite set. These sequences, called directed hyperedges, are the fundamental building blocks for hyperdigraph homology. In contrast, the building blocks for hypergraph homology are the sets of finite elements, while the building blocks for path homology are the directed paths.
The topological structures for hyperdigraph homology, hypergraph homology, and path homology are from the corresponding maximal chain complexes embedded in the spaces generated by their building blocks.
This section explores the hyperdigraph homology and topological hyperdigraph  Laplacians for hyperdigraphs. We follow the technique from topological hypergraphs and path complexes using sequences to construct chain complexes for hyperdigraph homology.
Our approach is inspired by  path homology \cite{grigor2012homologies,grigor2017homologies,grigor2020path} and embedded homology of hypergraphs \cite{bressan2019embedded}.

  \subsection{Hyperdigraph homology}
  %Let $V$ be a finite nonempty ordered set. Let $\mathbf{P}(V)$ be the power set of $V$. Note that $\mathbf{P}(V)=\coprod\limits_{n=1}^{|V|}\mathbf{P}_{n}(V)$, where $\mathbf{P}_{n}$ is the set of subsets with $n$ elements in $V$.
  A \emph{hyperdigraph} $\vec{\mathcal{H}}=(V,\vec{E})$ on $V$ is a pair such that $\vec{E}$ is a subset of $\mathbf{S}(V)=\coprod\limits_{n=1}^{|V|}\Sigma_{n}\times \mathbf{P}_{n}(V)$. Here, $\Sigma_{n}$ is the permutation group of order $n$. An element $(\sigma,e)\in \Sigma_{n}\times \mathbf{P}_{n}(V)$ is called a \emph{$\sigma$-directed $(n-1)$-hyperedge}. In particular, if all the directed hyperedges are trivially directed, i.e., the corresponding permutations are trivial permutations, then we say $\vec{\mathcal{H}}$ is \emph{undirected} or simply a hypergraph as usual.
%There are various definitions of the hyperdigraph \cite{ausiello2017directed,thakur2009linear}. Besides, different terms, such as directed hypergraphs and oriented hypergraphs, have appeared to refer to hyperdigraphs. In this paper, we follow the definition of the hyperdigraph in the reference \cite{berge1984hypergraphs} and study the topological homology and Laplacian for hyperdigraphs. The idea is mainly from the path homology \cite{grigor2012homologies,grigor2017homologies,grigor2020path} and the embedded homology of hypergraphs \cite{bressan2019embedded}.

 Let $\vec{\mathcal{H}}=(V,\vec{E}),\vec{\mathcal{H}}'=(V,\vec{E}')$ be two hyperdigraphs.
  A \emph{morphism of hyperdigraphs} $f:\vec{\mathcal{H}}\to \vec{\mathcal{H}}'$ is map $f:V\to V'$ such that $(\sigma,f(e))\in \vec{E}'$ for any $(\sigma,e)\in \vec{E}$.

  Let $X$ be an $n$-element ordered set. Let $\Sigma_{n} X$ be the set of permutations of elements in $X$.
  \begin{lemma}\label{formula:isomorphism}
  There is a bijection $\Sigma_{n}\times \{X\}\to \Sigma_{n}X$ between the set $\Sigma_{n}\times \{X\}$ and the set $\Sigma_{n}X$ given by
  \begin{equation*}
    \Sigma_{n}\times \{X\}\to \Sigma_{n}X,\quad (\sigma,\{X\})\mapsto (x_{\sigma(1)},x_{\sigma(2)},\dots,x_{\sigma(n)})
  \end{equation*}
  for $X=\{x_{1},x_{2},\dots,x_{n}\}$ and $\sigma\in \Sigma_{n}$.
  \end{lemma}
  \begin{proof}
  The surjectivity is from the fact that each sequence $(x_{k_{1}},x_{k_{2}},\dots,x_{k_{n}})$ in $\Sigma_{n}X$ can be uniquely written as the permutation $\sigma$ action on the sequence $(x_{1},x_{2},\dots,x_{n})$. On the other hand, if $(x_{\sigma(1)},x_{\sigma(2)},\dots,x_{\sigma(n)})=(x_{\tau(1)},x_{\sigma(2)},\dots,x_{\tau(n)})$, the permutation group action shows that $\tau^{-1}\sigma=1$. Here, $1$ is the identity element. Thus we have $\sigma=\tau$, the injectivity follows.
  \end{proof}

  \begin{remark}
  We can also use the ordered sequence to define   hyperdigraph homology. Precisely,
  a hyperdigraph $\vec{\mathcal{H}}=(V,\vec{E})$ on $V$ is a pair such that $\vec{E}$ is a subset of the set of sequences of distinct elements in $V$. A directed $p$-hyperedge of $\mathcal{H}$ is a sequence $v_{0}v_{1}\cdots v_{p}$ of distinct elements in $V$. This definition is equivalent to the previous definition of the hyperdigraph.
  \end{remark}
  Recall that a $(p,q)$-shuffle permutation is a permutation $\sigma$ in $\Sigma_{p+q}$ such that
  \begin{equation*}
    \sigma(1)<\cdots<\sigma(p),\quad \sigma(p)<\cdots<\sigma(p+q).
  \end{equation*}
  Let $\vec{\mathcal{H}}=(V,\vec{E})$ be a hyperdigraph such that each directed hyperedge is equipped with a $(p,q)$-shuffle permutation. Then we call $\vec{\mathcal{H}}$ a \emph{shuffle-hyperdigraph}. In other words, for each directed hyperedge $(\sigma,e)$ of $\vec{\mathcal{H}}$, the permutation $\sigma$ is a $(p,q)$-shuffle permutation for some $p+q=|e|$. Note that the shuffle-hyperdigraph can lead to another definition of hyperdigraph \cite{gallo1993directed}.
  \begin{example}\label{example:acyclic}
  Recall that a directed acyclic graph (DAG) is a digraph containing no directed cycle \cite{harary2014graphical}.
  A path complex of an acyclic digraph is a hypergraph.
  A path complex such that each path has distinct vertices is a hyperdigraph. Let $(V,\mathcal{P})$ be a path complex given by $V=\{0,1,2\}$ and
  \begin{equation*}
    \mathcal{P}=\{e_{0},e_{1},e_{2},e_{01},e_{12},e_{20},e_{012}\}.
  \end{equation*}
  The path complex $(V,\mathcal{P})$ can be identified with the hyperdigraph $(V,\vec{E})$, where
  \begin{equation*}
    \vec{E}=\{(0),(1),(2),(0,1),(1,2),(2,0),(0,1,2)\}.
  \end{equation*}
  \end{example}
  Now, let $S_{p}(V;\mathbb{K})$ be the $\mathbb{K}$-linear space generated by the elements in $\Sigma_{p+1}\times \mathbf{P}_{p+1}(V)$. Then $S_{\ast}(V;\mathbb{K})=(S_{p}(V;\mathbb{K}))_{p\geq 0}$ is a chain complex with the boundary operator $d_{p}:S_{p}(V;\mathbb{K})\to S_{p-1}(V;\mathbb{K})$ given by $d_{p}=\sum\limits_{i=0}^{p}(-1)^{i}\partial_{i}$. Here,
  \begin{equation*}
    \partial_{i}(x_{0},x_{1},\dots,x_{p})=(x_{0},\dots,\widehat{x_{i}},\dots,x_{p}),\quad 0\leq i\leq p,
  \end{equation*}
  where $(x_{0},x_{1},\dots,x_{p})$ is a sequence in $\Sigma_{p+1}V$ as given in Lemma \ref{formula:isomorphism}. In particular, we make the convention $d_{0}e=0$ for each $e\in V$. It can be verified that $d_{p-1}\circ d_{p}=0$ on $S_{\ast}(V;\mathbb{K})$ for $p\geq 1$.
  Let $F_{p}(\vec{\mathcal{H}};\mathbb{K})$ be the $\mathbb{K}$-linear space generated by $\vec{E}^{p}$. Here, $\vec{E}^{p}\subseteq \Sigma_{p+1}\times \mathbf{P}_{p+1}(V)$ is the set of directed $p$-hyperedges of $\vec{\mathcal{H}}$. Obviously, $F_{p}(\vec{\mathcal{H}};\mathbb{K})$ is a subspace of $S_{p}(V;\mathbb{K})$. Recall that $S_{\ast}(V;\mathbb{K})=(S_{p}(V;\mathbb{K}))_{p\geq 0}$ is a chain complex. The boundary operator $d_{p}$ on $S_{p}(V;\mathbb{K})$ can restrict to a map $d_{p}:F_{p}(\vec{\mathcal{H}};\mathbb{K})\to S_{p-1}(V;\mathbb{K})$.
  We denote
  \begin{equation}\label{equation:omega_chaincomplex}
    \Omega_{p}(\vec{\mathcal{H}};\mathbb{K})=F_{p}(\vec{\mathcal{H}};\mathbb{K})\cap d^{-1}F_{p-1}(\vec{\mathcal{H}};\mathbb{K})=\{x\in F_{p}(\vec{\mathcal{H}};\mathbb{K})|dx\in F_{p-1}(\vec{\mathcal{H}};\mathbb{K})\}.
  \end{equation}
  Then $\Omega_{\ast}(\vec{\mathcal{H}};\mathbb{K})=(\Omega_{p}(\vec{\mathcal{H}};\mathbb{K}))_{p\geq 0}$ is a chain complex.
  \begin{definition}
  The \emph{embedded homology} of a hyperdigraph $\vec{\mathcal{H}}$ is defined by
  \begin{equation*}
    H_{p}(\vec{\mathcal{H}};\mathbb{K})=H_{p}(\Omega_{\ast}(\vec{\mathcal{H}};\mathbb{K})),\quad p\geq 0.
  \end{equation*}
  \end{definition}
  If $\vec{\mathcal{H}}=(V,E)$ is an (undirected) hyperdigraph, the above definition reduces to the embedded homology of hypergraphs.

  \begin{example}
  The embedded homology of topological hyperdigraphs can reflect more information than the embedded homology of topological hypergraphs. For example, consider the hyperdigraph $\vec{\mathcal{H}}=(V,\vec{E})$, where $V=\{0,1\}$ and $\vec{E}=\{(0,1),(1,0)\}$. A straightforward calculation shows that
  \begin{equation*}
    H_{p}(\vec{\mathcal{H}};\mathbb{K})=\left\{
                                    \begin{array}{ll}
                                      \mathbb{K}, & \hbox{$p=1$;} \\
                                      0, & \hbox{otherwise.}
                                    \end{array}
                                  \right.
  \end{equation*}
  The cycle in $F_{\ast}(\vec{\mathcal{H}};\mathbb{K})$ is generated by $(0,1)+(1,0)$. However, the $1$-dimensional homology of any topological hypergraph with vertex set $\{0,1\}$ is trivial.
  \end{example}

  \begin{example}
  Example \ref{example:acyclic} continued. Let $\mathcal{P}$ be a path complex such that each path has distinct vertices. We can regard $\mathcal{P}$ as a hyperdigraph. Then the embedded homology of $\mathcal{P}$ coincides with the path homology of $\mathcal{P}$, that is,
  \begin{equation*}
    H_{p}^{emb}(\mathcal{P};\mathbb{K})=H_{p}^{path}(\mathcal{P};\mathbb{K}),\quad p\geq 0.
  \end{equation*}
  \end{example}

 Following a similar discussion as in \cite{grbic2022aspects}, we can establish the following proposition, which ensures that the persistent hyperdigraph homology introduced in Section \ref{section:persistence} is a persistence module.
  \begin{proposition}
  The homology $H_{\ast}(-;\mathbb{K}):\vec{\mathbf{H}}\mathbf{yper}\to \mathbf{Vec}_{\mathbb{K}}$ is a functor from the category of hyperdigraphs to the category of $\mathbb{K}$-linear spaces.
  \end{proposition}

  \subsection{From hyperdigraph homology to hypergraph homology}

  Now, let $\vec{\mathcal{H}}=(V,\vec{E})$ be a hyperdigraph. For each directed hyperedge $(\sigma,e)$, we have a subset $e$ of $\mathbf{P}(V)$. Thus we have a hypergraph $\overline{\mathcal{H}}=(V,E)$, where each hyperedge $e$ in $E$ is a projection of some directed hyperedge $(\sigma,e)$ in $\vec{E}$. We call $\overline{H}$ the \emph{reduced hypergraph} of the hyperdigraph $\vec{\mathcal{H}}$.
  Moreover, we have a map of $\mathbb{K}$-linear spaces
  \begin{equation*}
    \pi: F_{\ast}(\vec{\mathcal{H}};\mathbb{K})\to C_{\ast}(\overline{\mathcal{H}};\mathbb{K})
  \end{equation*}
  given by $\pi(\sigma,e)= (-1)^{\epsilon(\sigma)} e$. Here, $\epsilon(\sigma)$ is the Levi-Civita symbol of the permutation $\sigma$.  Then $\pi(\vec{\mathcal{H}})$ is a hypergraph. Thus we have the commutative diagram
  \begin{equation*}
      \xymatrix{
    F_{\ast}(\vec{\mathcal{H}};\mathbb{K})\ar@{->}[r]^{\pi}\ar@{^{(}->}[d]&C_{\ast}(\overline{\mathcal{H}};\mathbb{K})\ar@{^{(}->}[d]\\
    S_{\ast}(V;\mathbb{K})\ar@{->}[r]^{\pi}&C_{\ast}(V;\mathbb{K}).
    }
  \end{equation*}
  It induces a $\mathbb{K}$-linear map
  \begin{equation*}
    \pi:\Omega_{\ast}(\vec{\mathcal{H}};\mathbb{K})\to \mathrm{Inf}_{\ast}(\overline{\mathcal{H}};\mathbb{K}).
  \end{equation*}

  \begin{lemma}
  The map $\pi:\Omega_{\ast}(\vec{\mathcal{H}};\mathbb{K})\to \mathrm{Inf}_{\ast}(\overline{\mathcal{H}};\mathbb{K})$ is a chain map.
  \end{lemma}
  \begin{proof}
  It suffices to show the following diagram commutes.
  \begin{equation*}
    \xymatrix{
    F_{p}(\vec{\mathcal{H}};\mathbb{K})\ar@{->}[r]^{\pi}\ar@{->}[d]_{d_{p}}&C_{p}(\overline{\mathcal{H}};\mathbb{K})\ar@{->}[d]^{d_{p}}\\
    S_{p-1}(V;\mathbb{K})\ar@{->}[r]^{\pi}\ar@{->}[r]&C_{p-1}(V;\mathbb{K}).
    }
  \end{equation*}
  We will show $(-1)^{\sigma(i)}\partial_{\sigma(i)}\circ\pi(v_{\sigma(0)},v_{\sigma(1)},\dots,v_{\sigma(p)})=(-1)^{i}\pi\circ \partial_{i}(v_{\sigma(0)},v_{\sigma(1)},\dots,v_{\sigma(p)})$, which implies $\partial_{p}\circ\pi=\pi\circ \partial_{p}$. Here, $e=\{v_{0},v_{1},\dots,v_{p}\}$ is an ordered set. By a direct calculation, one has
  \begin{equation*}
    (-1)^{\sigma(i)}\partial_{\sigma(i)}\circ\pi (v_{\sigma(0)},v_{\sigma(1)},\dots,v_{\sigma(p)})=(-1)^{\epsilon(\sigma)+\sigma(i)}\partial_{\sigma(i)}e.
  \end{equation*}
  On the other hand, let $n_{i}(\sigma)$ be the sum of $|\{j<i|v_{\sigma(j)}>v_{\sigma(i)}\}|$ and $|\{j>i|v_{\sigma(j)}<v_{\sigma(i)}\}|$. Then we have
  \begin{equation*}
    \pi\circ \partial_{i}(v_{\sigma(0)},v_{\sigma(1)},\dots,v_{\sigma(p)})=(-1)^{i+\epsilon(\sigma)+n_{i}(\sigma)}\partial_{\sigma(i)}e.
  \end{equation*}
  If we fix $v_{\sigma(i)}$ in the sequence $v_{\sigma(0)},v_{\sigma(1)},\dots,v_{\sigma(p)}$ and permutate any two other elements $v_{\sigma(s)},v_{\sigma(t)}$, we obtain a sequence $v_{\tau(0)},v_{\sigma(1)},\dots,v_{\tau(p)}$ such that $\tau(j)=\sigma(j)$ for $j\neq s,t$ and $\tau(s)=\sigma(t),\tau(t)=\sigma(s)$. It is obvious that $n_{i}(\sigma)$ and $n_{i}(\tau)$ differ by an even number. This permutation does not change the signature. So the signature of $(v_{\sigma(0)},v_{\sigma(1)},\dots,v_{\sigma(p)})$ can be reduced to the signature of the sequence
  \begin{equation*}
    (v_{0},\dots,v_{i-1},v_{\sigma(i)},v_{i+1},\dots,v_{\sigma(i)-1},v_{i},v_{\sigma(i)+1},\dots,v_{p})
  \end{equation*}
  for $\sigma(i)>i$ or the sequence
  \begin{equation*}
    (v_{0},\dots,v_{\sigma(i)-1},v_{i},v_{\sigma(i)+1},\dots,v_{i-1},v_{\sigma(i)},v_{i+1},\dots,v_{p})
  \end{equation*}
  for $\sigma(i)<i$.
  It follows that $(-1)^{n_{i}(\sigma)}=(-1)^{\sigma(i)-i}$. Thus we have $\partial_{p}\circ\pi=\pi\circ \partial_{p}$.
  \end{proof}
  The above chain morphism induces a morphism of embedded homology.
  \begin{theorem}
  Let $\vec{\mathcal{H}}$ be a hyperdigraph. There is a natural morphism
  \begin{equation*}
    H_{\ast}(\pi):H_{\ast}(\vec{\mathcal{H}};\mathbb{K})\to H_{\ast}(\overline{\mathcal{H}};\mathbb{K}).
  \end{equation*}
  \end{theorem}
  \begin{proof}
  Let $f:\mathcal{H}\to \mathcal{H}'$ be a morphism of hyperdigraphs. Then there is $\mathbb{K}$-linear maps
  \begin{equation*}
    F_{p}(f):F_{p}(\vec{\mathcal{H}};\mathbb{K})\to F_{p}(\vec{\mathcal{H}}';\mathbb{K}),\quad (\sigma,e)\mapsto (\sigma,f(e))
  \end{equation*}
  and
  \begin{equation*}
    C_{p}(\overline{f}):C_{p}(\overline{\mathcal{H}};\mathbb{K})\to C_{p}(\overline{\mathcal{H}'};\mathbb{K}),\quad e\mapsto f(e).
  \end{equation*}
  Consider the following diagram
  \begin{equation*}
    \xymatrix{
    F_{p}(\vec{\mathcal{H}};\mathbb{K})\ar@{->}[r]^{\pi}\ar@{->}[d]_{F_{p}(f)}&C_{p}(\overline{\mathcal{H}};\mathbb{K})\ar@{->}[d]^{C_{p}(\overline{f})}\\
    F_{p}(\vec{\mathcal{H}}';\mathbb{K})\ar@{->}[r]^{\pi}\ar@{->}[r]&C_{p}(\overline{\mathcal{H}'};\mathbb{K}).
    }
  \end{equation*}
  For any $(\sigma,e)\in \vec{E}$, one has
  \begin{equation*}
    C_{p}(\overline{f})\circ\pi(\sigma,e)=(-1)^{\epsilon(\sigma)}f(e)=\pi\circ F_{p}(f)(\sigma,e).
  \end{equation*}
  Thus we obtain $C_{p}(\overline{f})\circ\pi=\pi\circ F_{p}(f)$. Since the embedded hyperdigraph homology $H_{\ast}$ is a functor, we have the desired result.
  \end{proof}
  Generally, the map $H_{\ast}(\pi)$ is not necessarily a surjection or an injection.
  \begin{example}
  Let $\vec{\mathcal{H}}_{1}=(V_{1},\vec{E}_{1})$ be a hyperdigraph with $V_{1}=\{1,2,3\}$ and
  \begin{equation*}
    \vec{E}_{1}=\{(1,2),(1,3),(3,2),(1,2,3)\}.
  \end{equation*}
  It follows that $H_{p}(\vec{\mathcal{H}}_{1};\mathbb{K})=\left\{
                                                   \begin{array}{ll}
                                                     \mathbb{K}, & \hbox{$p=1$;} \\
                                                     0, & \hbox{otherwise.}
                                                   \end{array}
                                                 \right.
  $ and $H_{p}(\overline{\mathcal{H}_{1}};\mathbb{K})=0$. Thus $H_{\ast}(\pi):H_{\ast}(\vec{\mathcal{H}}_{1};\mathbb{K})\to H_{\ast}(\overline{\mathcal{H}_{1}};\mathbb{K})$ is not an injection.
  Let $\vec{\mathcal{H}}_{2}=(V_{2},\vec{E}_{2})$ be a hyperdigraph with $V_{2}=\{1,2,3,4\}$ and
  \begin{equation*}
    \vec{E}_{2}=\{(1,2,3),(2,3,4),(3,4,1),(4,1,2)\}.
  \end{equation*}
  Then we have $H_{p}(\vec{\mathcal{H}}_{2};\mathbb{K})=0$ and $H_{p}(\overline{\mathcal{H}_{2}};\mathbb{K})=\left\{
                                                   \begin{array}{ll}
                                                     \mathbb{K}, & \hbox{$p=2$;} \\
                                                     0, & \hbox{otherwise.}
                                                   \end{array}
                                                 \right.
  $ Then the map $H_{\ast}(\pi):H_{\ast}(\vec{\mathcal{H}}_{2};\mathbb{K})\to H_{\ast}(\overline{\mathcal{H}_{2}};\mathbb{K})$ is not a surjection.
  \end{example}

  \subsection{Topological hyperdigraph Laplacians}
%  In this section, we introduce the hyperdigraph Laplacian from a topological view. There are high dimensional structures on the hyperdigraph Laplacians. Moreover, the hyperdigraph Laplacian can be developed into a topological persistence to be applied to the data analysis. A few examples and calculations will be shown to readers.
%
  In this section, we will introduce topological hyperdigraph Laplacians as  a mathematical tool that can be used to study high-dimensional topological structures. Furthermore, it can be used to develop a topological persistence that can be applied to data analysis. We will provide some examples and calculations to illustrate these concepts for readers. From now on, we always take $\mathbb{K}=\mathbb{R}$.

  Let $\vec{\mathcal{H}}=(V,\vec{E})$ be a hyperdigraph. We have an infimum complex
  \begin{equation*}
    \cdots\to \Omega_{p+1}(\vec{\mathcal{H}};\mathbb{R})\stackrel{d_{p+1}}{\to} \Omega_{p}(\vec{\mathcal{H}};\mathbb{R})\stackrel{d_{p}}{\to} \Omega_{p-1}(\vec{\mathcal{H}};\mathbb{R})\to \cdots.
  \end{equation*}
  We regard the collection of directed hyperedges $\vec{E}_{1},\vec{E}_{2},\cdots,\vec{E}_{n}$ as a standard basis of the $\mathbb{K}$-linear space $F_{\ast}(\vec{\mathcal{H}};\mathbb{R})$. Here, $n=|\vec{E}|$ is the number of directed hyperedges of $\vec{\mathcal{H}}$. Let $v_{1},v_{2},\dots,v_{N}$ be a standard basis of $\mathbb{K}[V]$. Here, $N=|V|$. Endow $F_{\ast}(\vec{\mathcal{H}};\mathbb{R})$ and $\mathbb{K}[V]$ with the standard inner products, respectively, given by
  \begin{equation*}
    \langle   v_{i},  v_{j}\rangle =\left\{
                      \begin{array}{ll}
                        1, & \hbox{$i=j$;} \\
                        0, & \hbox{otherwise.}
                      \end{array}
                    \right.\quad   \langle   \vec{E}_{i},  \vec{E}_{j}\rangle =\left\{
                      \begin{array}{ll}
                        1, & \hbox{$i=j$;} \\
                        0, & \hbox{otherwise.}
                      \end{array}
                    \right.
  \end{equation*}
  Then the chain complex $\Omega_{\ast}(\vec{\mathcal{H}};\mathbb{R})$ inherits the inner product of $F_{\ast}(\vec{\mathcal{H}};\mathbb{R})$ as subspace. The adjoint functor of $d_{p}:\Omega_{p}(\vec{\mathcal{H}};\mathbb{R})\to \Omega_{p-1}(\vec{\mathcal{H}};\mathbb{R})$ is a $\mathbb{K}$-linear map $(d_{p})^{\ast}:\Omega_{p-1}(\vec{\mathcal{H}};\mathbb{R})\to \Omega_{p}(\vec{\mathcal{H}};\mathbb{R})$ such that
  \begin{equation*}
    \langle d_{p}x,y\rangle=\langle x,(d_{p})^{\ast}y\rangle
  \end{equation*}
  for any $x\in \Omega_{p}(\vec{\mathcal{H}};\mathbb{R})$ and $y\in \Omega_{p-1}(\vec{\mathcal{H}};\mathbb{R})$.

  \begin{definition}
  The \emph{hyperdigraph Laplacian} $\Delta^{\vec{\mathcal{H}}}_{p}:\Omega_{p}(\vec{\mathcal{H}};\mathbb{R})\to \Omega_{p}(\vec{\mathcal{H}};\mathbb{R})$ of $\vec{\mathcal{H}}$ is defined by
  \begin{equation*}
    \Delta^{\vec{\mathcal{H}}}_{p}=(d_{p})^{\ast}\circ d_{p}+d_{p+1}\circ (d_{p+1})^{\ast},\quad p\geq 0.
  \end{equation*}
  \end{definition}
  %We choose a family of standard orthogonal basis of $\Omega_{p}(\mathcal{H};\mathbb{R})$ for $p\geq 0$. Let $B_{p}$ be the representation matrix of $d_{p}$ with respect to the chosen standard basis. The the representation matrix of $\Delta^{\mathcal{H}}_{p}$ is given by
  %\begin{equation*}
  %  L_{p}=B_{p}B_{p}^{T}+B_{p+1}^{T}B_{p+1}.
  %\end{equation*}
  %The following proposition shows that all the eigenvalues of $\Delta^{\mathcal{H}}_{p}$ is non-negative.
  Similar to the discussion in Section \ref{section:hypergraph}, we have the following propositions.
  \begin{proposition}
  The operator $\Delta^{\vec{\mathcal{H}}}_{p}$ is self-adjoint and non-negative definite.
  \end{proposition}
  %\begin{proof}
  %Obviously, for any $x,y\in\Omega_{p}(\mathcal{H};\mathbb{R})$, we have
  %\begin{equation*}
  %  \langle \Delta^{\mathcal{H}}_{p}x,y\rangle=\langle d_{p}x,d_{p}y\rangle+\langle (d_{p+1})^{\ast}x,(d_{p+1})^{\ast}y\rangle=\langle x,\Delta^{\mathcal{H}}_{p}y\rangle.
  %\end{equation*}
  %On the other hand, one has
  %\begin{equation*}
  %  \langle \Delta^{\mathcal{H}}_{p}x,x\rangle=\langle d_{p}x,d_{p}x\rangle+\langle (d_{p+1})^{\ast}x,(d_{p+1})^{\ast}x\rangle\geq 0.
  %\end{equation*}
  %Thus $\Delta^{\mathcal{H}}_{p}$ is self-adjoint and non-negative definite.
  %\end{proof}
  %By the algebraic Hodge decomposition theorem, one has
  \begin{proposition}\label{proposition:decomposition1}
  $\Omega_{p}(\vec{\mathcal{H}};\mathbb{R})=\ker \Delta^{\vec{\mathcal{H}}}_{p}\oplus \mathrm{Im}d_{p}\oplus\mathrm{Im}(d_{p+1})^{\ast}$. Moreover, we have $\ker \Delta^{\vec{\mathcal{H}}}_{p}=\ker d_{p}\cap \ker (d_{p+1})^{\ast}\cong H_{p}(\vec{\mathcal{H}};\mathbb{R})$.
  \end{proposition}
  This above proposition shows that the number of zero eigenvalues of $\Delta^{\vec{\mathcal{H}}}_{p}$ equals to the $p$-th Betti number of $\vec{\mathcal{H}}$.
  \begin{theorem}
  Let $\vec{\mathcal{H}}=(V,\vec{E})$ be a hyperdigraph. Let $\lambda_{0}(k)$ be the $k$-th smallest Laplacian eigenvalue of $\Delta^{\vec{\mathcal{H}}}_{0}$. If $\vec{E}_{0}\neq \emptyset$, then
  \begin{enumerate}
    \item[$(i)$] $\lambda_{0}(1)=0$;
    \item[$(ii)$] $\lambda_{0}(2)>0$ if and only if for any directed $0$-hyperedges $(v),(w)$ of $\vec{\mathcal{H}}$, there is a path of $1$-hyperedges from $(v)$ to $(w)$.
  \end{enumerate}
  \end{theorem}
  \begin{proof}
  $(i)$  Since there exists at least one directed $0$-hyperedge in $\vec{E}$, we have $\Omega_{0}(\vec{\mathcal{H}};\mathbb{R})=F_{0}(\vec{\mathcal{H}};\mathbb{R})$. Moreover, the representation matrix of $\Delta^{\vec{\mathcal{H}}}_{0}:\Omega_{0}(\vec{\mathcal{H}};\mathbb{R})\to \Omega_{0}(\vec{\mathcal{H}};\mathbb{R})$ is an $r\times r$ matrix. Here, $r=\dim \Omega_{0}(\vec{\mathcal{H}};\mathbb{R})\geq 1$. Consider the augmented chain complex
  \begin{equation*}
   \cdots\to \Omega_{2}(\vec{\mathcal{H}};\mathbb{R})\stackrel{d_{2}}{\to} \Omega_{1}(\vec{\mathcal{H}};\mathbb{R})\stackrel{d_{1}}{\to} \Omega_{0}(\vec{\mathcal{H}};\mathbb{R})\stackrel{\tilde{d}_{0}}{\to} \mathbb{K}\to 0,
  \end{equation*}
  where $\tilde{d}(\sum\limits_{j=1}^{m}a_{j}\cdot(v_{j}))=\sum\limits_{j=1}^{m}a_{j}$ for $\sum\limits_{j=1}^{m}a_{j}\cdot(v_{j})\in \Omega_{0}(\vec{\mathcal{H}};\mathbb{R})$. It follows that
  \begin{equation*}
    \im d_{1}\subseteq \ker \tilde{d}_{0}.
  \end{equation*}
  Choose a directed $0$-hyperedge $(v)\in \vec{E}$. Observing that $\ker \tilde{d}_{0}(v)\neq 0$. One has $\im d_{1}\subsetneq \Omega_{0}(\vec{\mathcal{H}};\mathbb{R})$. It follows that
  \begin{equation*}
    H_{0}(\vec{\mathcal{H}};\mathbb{R})\neq 0.
  \end{equation*}
  By Proposition \ref{proposition:decomposition1}, $\dim\ker\Delta^{\vec{\mathcal{H}}}_{0}=\dim H_{0}(\vec{\mathcal{H}};\mathbb{R})>0$. It follows that $\lambda_{0}(1)=0$.

  $(ii)$ If $\lambda_{0}(2)>0$, one has $\dim H_{0}(\vec{\mathcal{H}};\mathbb{R})=\dim\ker\Delta^{\vec{\mathcal{H}}}_{0}=1$. So for any $(v),(w)\in \vec{E}^{0}$, we have $(v)-(w)\in \im d_{1}$. Suppose there is no path of $1$-hyperedges from $(v)$ to $(w)$. Let $(v)-(w)=d_{1}\sum\limits_{j=1}^{k}a_{i}(v_{j}w_{j})$. Consider the function $f:\Omega_{0}(\vec{\mathcal{H}};\mathbb{R})\to \mathbb{R}$ given by
  \begin{equation*}
    f((x))=\left\{
             \begin{array}{ll}
               0, & \hbox{if there is a path from $(v)$ to $(x)$;} \\
               1, & \hbox{otherwise.}
             \end{array}
           \right.
  \end{equation*}
  Then we have
  \begin{equation*}
  -1=f((v)-(w)))=\sum\limits_{j=1}^{k}a_{i}(f(w_{j})-f(v_{j})).
  \end{equation*}
  Note that $f(w_{j})=f(v_{j})$ since $(v_{j}w_{j})$ is a directed $1$-hyperedge. One has $\sum\limits_{j=1}^{k}a_{i}(f(w_{j})-f(v_{j}))=0$, contradiction. So there is a path of $1$-hyperedges from $(v)$ to $(w)$.
  \end{proof}

  \begin{example}
  Let $\vec{\mathcal{H}}=(V,\vec{E})$ be a hyperdigraph with $V=\{1,2,3\}$ and
  \begin{equation*}
    \vec{E}=\{(1),(2),(1,2),(1,3),(3,2),(1,2,3)\}.
  \end{equation*}
  Then we have
  \begin{eqnarray*}
    F_{0}(\vec{\mathcal{H}};\mathbb{R}) &=& \mathrm{span}\{(1),(2)\}, \\
    F_{1}(\vec{\mathcal{H}};\mathbb{R}) &=& \mathrm{span}\{(1,2),(1,3),(3,2)\}, \\
    F_{2}(\vec{\mathcal{H}};\mathbb{R}) &=& \mathrm{span}\{(1,2,3)\}.
  \end{eqnarray*}
  Note that $d(1,2,3)=(1,2)-(1,3)+(2,3)\notin F_{1}(\vec{\mathcal{H}};\mathbb{R})$. By Eq. (\ref{equation:omega_chaincomplex}), we obatin
  \begin{eqnarray*}
    \Omega_{0}(\vec{\mathcal{H}};\mathbb{R}) &=& \mathrm{span}\{(1),(2)\}, \\
    \Omega_{1}(\vec{\mathcal{H}};\mathbb{R}) &=& \mathrm{span}\{(1,2),(1,3)+(3,2)\}, \\
    \Omega_{2}(\vec{\mathcal{H}};\mathbb{R}) &=& 0.
  \end{eqnarray*}
  Now, we choose the standard orthogonal basis of $\Omega_{\ast}(\vec{\mathcal{H}};\mathbb{R})$ as $(1),(2),(1,2),\frac{(1,3)+(3,2)}{\sqrt{2}}$. Then the boundary matrices are given by
  \begin{eqnarray*}
    d_0\left(          \begin{array}{c}
                       (1) \\
                       (2) \\
                     \end{array}
                     \right)&=&\left(
                       \begin{array}{cc}
                        0\\ 0\\
                       \end{array}
                     \right), \\
    d_{1}\left(
           \begin{array}{c}
             (1,2) \\
             \frac{(1,3)+(3,2)}{\sqrt{2}} \\
           \end{array}
         \right)&=&\left(
                   \begin{array}{cc}
                     -1 & 1 \\
                     -\frac{1}{\sqrt{2}} & \frac{1}{\sqrt{2}} \\
                   \end{array}
                 \right)\left(
           \begin{array}{c}
             (1) \\
             (2) \\
           \end{array}
         \right) .
  \end{eqnarray*}
  Note that $\mathrm{rank} \left(
                   \begin{array}{cc}
                     -1 & 1 \\
                     -\frac{1}{\sqrt{2}} & \frac{1}{\sqrt{2}} \\
                   \end{array}
                 \right)=1$.
  We have the hyperdigraph homology $H_{p}(\vec{\mathcal{H}};\mathbb{R})=\left\{
                                                                                 \begin{array}{ll}
                                                                                   \mathbb{K}, & \hbox{$p=0,1$;} \\
                                                                                   0, & \hbox{otherwise.}
                                                                                 \end{array}
                                                                               \right.
  $
  Moreover, the representation matrix of $\Delta^{\vec{\mathcal{H}}}_{p}$ with respect to the chosen basis is given by
  \begin{equation*}
    L^{\vec{\mathcal{H}}}_{0}= \left(
                   \begin{array}{cc}
                     -1 & 1 \\
                     -\frac{1}{\sqrt{2}} & \frac{1}{\sqrt{2}} \\
                   \end{array}
                 \right)^{T}\left(
                   \begin{array}{cc}
                     -1 & 1 \\
                     -\frac{1}{\sqrt{2}} & \frac{1}{\sqrt{2}} \\
                   \end{array}
                 \right)= \left(
                   \begin{array}{cc}
                     \frac{3}{2} & -\frac{3}{2} \\
                     -\frac{3}{2} & \frac{3}{2} \\
                   \end{array}
                 \right)
  \end{equation*}
  and
  \begin{equation*}
    L^{\vec{\mathcal{H}}}_{1}= \left(
                   \begin{array}{cc}
                     -1 & 1 \\
                     -\frac{1}{\sqrt{2}} & \frac{1}{\sqrt{2}} \\
                   \end{array}
                 \right)\left(
                   \begin{array}{cc}
                     -1 & 1 \\
                     -\frac{1}{\sqrt{2}} & \frac{1}{\sqrt{2}} \\
                   \end{array}
                 \right)^{T}= \left(
                   \begin{array}{cc}
                     2 & \sqrt{2} \\
                     \sqrt{2} & 1 \\
                   \end{array}
                 \right).
  \end{equation*}
  The spectrum of $\Delta_{0}^{\vec{\mathcal{H}}}$ is $\{0,3\}$, the number of zero eigenvalues is exactly $1$. Similarly, the spectrum of $\Delta_{1}^{\vec{\mathcal{H}}}$ is $\{0,3\}$, which is consistent with the Betti number. Note that if we choose the standard orthogonal basis of $\Omega_{\ast}(\vec{\mathcal{H}};\mathbb{R})$ as $\frac{(1)+(2)}{\sqrt{2}},\frac{(1)-(2)}{\sqrt{2}},(1,2),\frac{(1,3)+(3,2)}{\sqrt{2}}$, one can obtain the corresponding representation matrices of Laplacians as
  \begin{equation*}
    \tilde{L}^{\vec{\mathcal{H}}}_{0}= \left(
                           \begin{array}{cc}
                             0 & 0 \\
                             0 & 3 \\
                           \end{array}
                         \right),\quad   \tilde{L}^{\vec{\mathcal{H}}}_{1}= \left(
                           \begin{array}{cc}
                             2 & \sqrt{2} \\
                             \sqrt{2} & 1 \\
                           \end{array}
                         \right).
  \end{equation*}
  In fact, the eigenvalues of the above matrices coincides with the previous ones.
  \end{example}

  \begin{figure}[!ht]
    \centering
    \includegraphics[width=14cm]{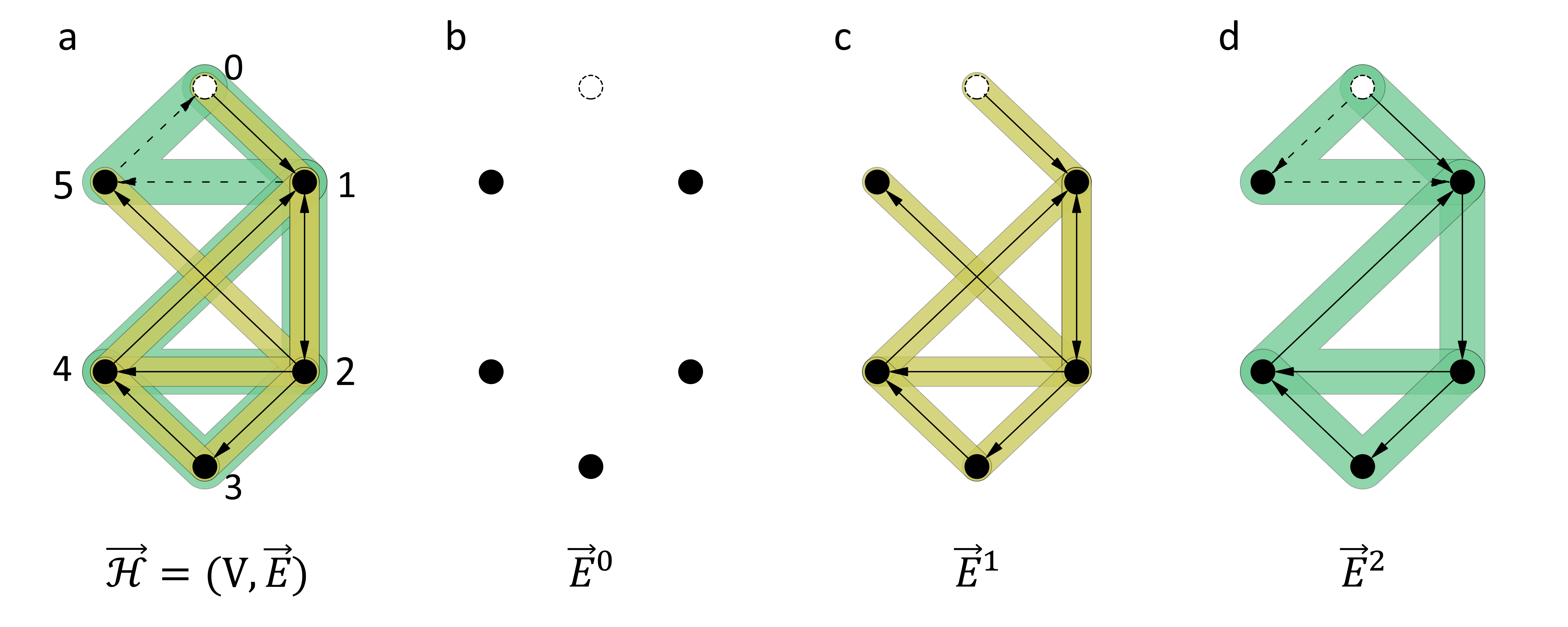}
    % \captionsetup{width=16cm}
    \caption{
      {\bf a} Illustration of hyperdigraph $\vec{\mathcal{H}}$ in Example \ref{example:hyperdigraph2}.
      {\bf b}, {\bf c}, and {\bf d} Illustration of directed 0-hyperedges, directed 1-hyperedges, and directed 2-hyperedges, respectively. The solid black vertices indicate 0-hyperedges and the dashed hollow circle indicates a vertex that is not a directed 0-hyperedge. The yellow background combined with directed edges represents directed 1-hyperedges. The green background combined with 2 consecutive directed edges represents directed 2-hyperedges. Here, the dashed arrows indicate the directed edges that do not exist.}
    \label{figure:hyperdigraph}
  \end{figure}

  \begin{example}\label{example:hyperdigraph2}
    As shown in Figure \ref{figure:hyperdigraph}{\bf a}, let $\mathcal{\vec{H}}=(V,\vec{E})$ be a hyperdigraph with vertex set $V=\{0, 1, 2, 3, 4, 5\}$ and the directed hyperedge set $\vec{E}$, where $\vec{E}$ contains the directed 0-hyperedges (Figure \ref{figure:hyperdigraph}{\bf b}), directed 1-hyperedges (Figure \ref{figure:hyperdigraph}{\bf c}), and the directed 2-hyperedges (Figure \ref{figure:hyperdigraph}{\bf d}) as follows,
    \begin{eqnarray*}
      \vec{E}^0&=&\{(1), (2), (3), (4), (5)\}, \\
      \vec{E}^1&=&\{(0, 1), (1, 2), (2, 1), (2, 3),  (2, 4), (2, 5),  (3, 4), (4, 1)  \}, \\
      \vec{E}^2&=&\{(0, 1, 2), (0, 5, 1),  (2, 3, 4), (2, 4, 1)\}.
    \end{eqnarray*}
    So the standard orthogonal basis of $\Omega_{\ast}(\vec{\mathcal{H}};\mathbb{R})$ can be chosen as follows,
    \begin{equation*}
      (1), (2), (3), (4), (5), (1, 2), (2, 1), (2, 3), (2, 4), (2, 5), (3, 4), (4, 1), (2, 3, 4), (2, 4, 1).
    \end{equation*}
    Then, the representation matrices of the boundary operators $d_0$,$d_1$, and $d_2$ are given as follows,
    \begin{eqnarray*}
      d_0\left(          \begin{array}{c}
                       (1) \\
                       (2) \\
                       (3) \\
                       (4) \\
                       (5) \\
                     \end{array}
                     \right)&=&\left(
                       \begin{array}{cc}
                        0\\ 0\\ 0\\ 0\\ 0 \\
                       \end{array}
                     \right),\\
      d_{1}\left(
               \begin{array}{c}
                (1, 2)\\(2, 1)\\(2, 3)\\(2, 4)\\(2, 5)\\(3, 4)\\(4, 1)\\
               \end{array}
             \right) &=& \left(
                       \begin{array}{ccccccc}
                       -1& 1& 0& 0&0\\
                        1& 1& 0& 0&0\\
                        0&-1& 1& 0&0\\
                        0&-1& 0& 1&0\\
                        0&-1& 0& 0&1\\
                        0& 0&-1& 1&0\\
                        1& 0& 0&-1&0\\
                       \end{array}
                     \right)\left(
               \begin{array}{c}
               (1)\\(2)\\(3)\\(4)\\(5)\\
               \end{array}
             \right), \\
      d_{2}\left(
        \begin{array}{c}
          (2, 3, 4)\\(2, 4, 1)\\
        \end{array}
      \right)&=&\left(
        \begin{array}{ccccccc}
        0&0&1&-1&0&1&0\\
        0&-1&0&1&0&0&1\\
        \end{array}
      \right)\left(
        \begin{array}{c}
          (1, 2)\\(2, 1)\\(2, 3)\\(2, 4)\\(2, 5)\\(3, 4)\\(4, 1)\\
        \end{array}
      \right).
    \end{eqnarray*}
    The representation matrices of the Laplacian $\Delta^{\vec{\mathcal{H}}}_{p}$ with respect to the chosen basis are listed as follows,
    \begin{eqnarray*}
      L_{0}^{\mathcal{\vec{H}}}&=&\left(
                     \begin{array}{ccccc}
                       3&-2& 0&-1& 0\\
                      -2& 5&-1&-1&-1\\
                       0&-1& 2&-1& 0\\
                      -1&-1&-1& 3& 0\\
                       0&-1& 0& 0& 1\\
                     \end{array}
                   \right), \\
      L_{1}^{\mathcal{\vec{H}}}&=&\left(
                    \begin{array}{ccccccc}
                       2& -2& -1& -1& -1& 0&-1\\
                      -2&  3&  1&  0&  1& 0& 0\\
                      -1&  1&  3&  0&  1& 0& 0\\
                      -1&  0&  0&  4&  1& 0& 0\\
                      -1&  1&  1&  1&  2& 0& 0\\
                       0&  0&  0&  0&  0& 3&-1\\
                      -1&  0&  0&  0&  0&-1& 3\\
                    \end{array}
                  \right), \\
      L_{2}^{\mathcal{\vec{H}}}&=&\left(
                    \begin{array}{cc}
                      3& -1\\
                      -1& 3\\
                    \end{array}
                  \right).
    \end{eqnarray*}

  Then the spectra of the Laplacian matrices can be generated, which are $\mathbf{Spec}(L_{0}^{\mathcal{\vec{H}}})=\{0, 1.044$, $2.332, 4.080, 6.544\}$, $\mathbf{Spec}(L_{1}^{\mathcal{\vec{H}}})=\{0, 1.044, 2, 2.332, 4, 4.080, 6.544\}$, and $\mathbf{Spec}(L_{2}^{\mathcal{\vec{H}}})=\{2,4\}$. The corresponding Betti numbers are $\beta_0=1,\beta_1=1, $ and $ \beta_2=0$, and the smallest eigenvalues of the non-harmonic spectra are $\lambda_0=1.044,\lambda_1=1.044$, and $\lambda_2=2$, respectively.

  \end{example}

  There are various homologies and topological Laplacians on different objects. We list these objects, which are a family of discrete objects satisfying certain properties, and the corresponding topological Laplacians in Table \ref{table:graph_Lap}.
  Fix a nonempty finite ordered set $V$. Recall that $\mathbf{S}(V)$ is the set of all sequence with distinct elements in $V$. Let $\mathbf{\tilde{S}}(V)$ be the set of all sequence with elements in $V$.
  For a subset $K$ of $\mathbf{P}(V)$, let $\partial_{\ast}K=\{\tau\subset \sigma|\sigma\in K\}$. For a subset $P$ in $\mathbf{\tilde{S}}(V)$, let $\partial_{0}P=\{(x_{1},\dots,x_{p})|(x_{0},x_{1},\dots,x_{p})\in P\}$ and $\partial_{\infty}P=\{(x_{0},\dots,x_{p-1})|(x_{0},x_{1},\dots,x_{p})\in P\}$.
  A graph is a vertex set $V$ equipped with a collection of $1$-dimensional edges, that is, a subset of $\mathbf{P}_{2}(V)$. A simplicial complex can be regarded as a vertex set equipped with a subset $K$ of $\mathbf{P}(V)$ satisfying $\partial_{\ast}K\subseteq K$. A digraph is a vertex set $V$ equipped with an edge set $E$ satisfying $E\subseteq \tilde{\mathbf{S}}_{2}(V)$. A path complex can be thought as a vertex set $V$ equipped with a collection $P$ satisfying $\partial_{0}P,\partial_{\infty}P\subseteq P$. The hypergraph homology and hyperdigraph homology can be regarded as a vertex set $V$ equipped with the corresponding edge sets $E\subseteq \mathbf{P}(V)$ and $\vec{E}\subseteq \mathbf{S}(V)$, respectively.
  \begin{table}[h]
    \centering
    \caption{Topological Laplacians on different objects}\label{table:graph_Lap}
    \begin{tabular}{c|c|c|c}
      \hline
      Homologies & notation & restriction & topological Laplacians \\
    \hline
      graph & $G=(V,E)$, $E\subseteq \mathbf{P}_{2}(V)$ & none& Laplacian \\
      simplicial complex & $\mathcal{K}=(V,K)$, $K\subseteq \mathbf{P}(V)$ & $\partial_{\ast}K\subseteq K$& Laplacian  \\
      digraph & $G=(V,E)$, $E\subseteq \mathbf{S}_{2}(V)$ & none & path Laplacian  \\
      path complex & $\mathcal{P}=(V,P)$, $P\subseteq \mathbf{\tilde{S}}(V)$& $\partial_{0}P,\partial_{\infty}P\subseteq P$ & path Laplacian  \\
      hypergraph & $\mathcal{H}=(V,E)$, $E\subseteq \mathbf{P}(V)$ & none& hypergraph Laplacian  \\
      hyperdigraph & $\vec{\mathcal{H}}=(V,\vec{E})$,  $\vec{E}\subseteq \mathbf{S}(V)$ & none & hyperdigraph Laplacian  \\
      \hline
    \end{tabular}
  \end{table}

\subsection{The calculations of topological hyperdigraph Laplacians}
  In this section, we provide more examples to illustrate topological hyperdigraph Laplacians. Our algorithm is based on the following calculation process. If there is no ambiguity, we will denote the sequence $(v_{0},v_{1},\dots,v_{p})$ by $(v_{0}v_{1}\cdots v_{p})$ for convenience.

  \begin{example}
  Let $\vec{\mathcal{H}}=(V,\vec{E})$ be a hyperdigraph with $V=\{1,2,3,4,5\}$ and
  \begin{equation*}
    \vec{E}=\{(12),(25),(13),(35),(14),(45),(125),(135),(145)\}.
  \end{equation*}
  By a direct calculation, we have
  \begin{eqnarray*}
    \Omega_{0}(\vec{\mathcal{H}};\mathbb{R}) &=& \mathrm{span}\{(1),(2),(3),(4),(5)\}, \\
    \Omega_{1}(\vec{\mathcal{H}};\mathbb{R}) &=& \mathrm{span}\{(12),(25),(13),(35),(14),(45)\}, \\
    \Omega_{2}(\vec{\mathcal{H}};\mathbb{R}) &=& \mathrm{span}\{(125)-(135),(125)-(145)\}.
  \end{eqnarray*}
  We choose the standard orthogonal basis of $\Omega_{\ast}(\vec{\mathcal{H}};\mathbb{R})$ as
  \begin{equation*}
    (1),(2),(3),(4),(5),(12),(25),(13),(35),(14),(45),(x),(y).
  \end{equation*}
  Here, $(x)=\frac{1}{\sqrt{2}}(125)-\frac{1}{\sqrt{2}}(135)$ and $(y)=-\frac{1}{\sqrt{6}}(125)-\frac{1}{\sqrt{6}}(135)+\frac{2}{\sqrt{6}}(145)$.
  Then the boundary matrices are given by
  \begin{equation*}
    d_0\left(          \begin{array}{c}
                       (1) \\
                       (2) \\
                       (3) \\
                       (4) \\
                       (5) \\
                     \end{array}
                     \right)=\left(
                       \begin{array}{cc}
                        0\\ 0\\ 0\\ 0\\ 0 \\
                       \end{array}
                     \right),\quad d_{1} \left(\begin{array}{c}
    (12) \\
    (25) \\
    (13) \\
    (35) \\
    (14) \\
    (45) \\
    \end{array}
    \right)=\left(
              \begin{array}{ccccc}
                -1 & 1 & 0 & 0 & 0 \\
                0 & -1 & 0 & 0 & 1 \\
                -1 & 0 & 1 & 0 & 0 \\
                0 & 0 & -1 & 0 & 1 \\
                -1 & 0 & 0 & 1 & 0 \\
                0 & 0 & 0 & -1 & 1 \\
              \end{array}
            \right)\left(
                     \begin{array}{c}
                       (1) \\
                       (2) \\
                       (3) \\
                       (4) \\
                       (5) \\
                     \end{array}
                   \right),
  \end{equation*}
  \begin{equation*}
    d_{2}\left(
           \begin{array}{c}
             (x) \\
             (y) \\
           \end{array}
         \right)=\left(
                   \begin{array}{cccccc}
                     \frac{1}{\sqrt{2}} & \frac{1}{\sqrt{2}}  & -\frac{1}{\sqrt{2}} & -\frac{1}{\sqrt{2}}& 0 & 0 \\
                     -\frac{1}{\sqrt{6}} & -\frac{1}{\sqrt{6}} & -\frac{1}{\sqrt{6}} & -\frac{1}{\sqrt{6}} & \frac{2}{\sqrt{6}} & \frac{2}{\sqrt{6}}
                      \\
                   \end{array}
                 \right)\left(
                                                                                                                                                      \begin{array}{c}
                                                                                                                                                        (12) \\
                                                                                                                                                        (25) \\
                                                                                                                                                        (13) \\
                                                                                                                                                        (35) \\
                                                                                                                                                        (14) \\
                                                                                                                                                        (45) \\
                                                                                                                                                      \end{array}
                                                                                                                                                    \right).
  \end{equation*}
  Hence, the representation matrices of $\Delta_{0}^{\vec{\mathcal{H}}},\Delta_{1}^{\vec{\mathcal{H}}}, $and $ \Delta_{2}^{\vec{\mathcal{H}}}$ are
  \begin{equation*}
    L_{0}^{\vec{\mathcal{H}}}=\left(
                                \begin{array}{ccccc}
                                  3 & -1 & -1 & -1 & 0 \\
                                  -1 & 2 & 0 & 0 & -1 \\
                                  -1 & 0 & 2 & 0 & -1 \\
                                  -1 & 0 & 0 & 2 & -1 \\
                                  0 & -1 & -1 & -1 & 3 \\
                                \end{array}
                              \right),L_{1}^{\vec{\mathcal{H}}}=\left(
                                                                  \begin{array}{cccccc}
                                                                    \frac{8}{3} & -\frac{1}{3} & \frac{2}{3} & -\frac{1}{3} & \frac{2}{3} & -\frac{1}{3} \\
                                                                    -\frac{1}{3} & \frac{8}{3} & -\frac{1}{3} & \frac{2}{3} & -\frac{1}{3} & \frac{2}{3} \\
                                                                    \frac{2}{3} & -\frac{1}{3} & \frac{8}{3} & -\frac{1}{3} & \frac{2}{3} & -\frac{1}{3} \\
                                                                    -\frac{1}{3} & \frac{2}{3} & -\frac{1}{3} & \frac{8}{3} & -\frac{1}{3} & \frac{2}{3} \\
                                                                    \frac{2}{3} & -\frac{1}{3} & \frac{2}{3} & -\frac{1}{3} & \frac{8}{3} & -\frac{1}{3} \\
                                                                    -\frac{1}{3} & \frac{2}{3} & -\frac{1}{3} & \frac{2}{3} & -\frac{1}{3} & \frac{8}{3} \\
                                                                  \end{array}
                                                                \right)
  \end{equation*}
  and
  \begin{equation*}
    L_{2}^{\vec{\mathcal{H}}}=\left(
                                  \begin{array}{cc}
                                    2 & 0 \\
                                    0 & 2 \\
                                  \end{array}
                                \right).
  \end{equation*}
  The spectra of $\Delta_{0}^{\vec{\mathcal{H}}},\Delta_{1}^{\vec{\mathcal{H}}}, $and $ \Delta_{2}^{\vec{\mathcal{H}}}$ are $\{0,2,2,3,5\}$, $\{2,2,2,2,3,5\}$, and $\{2,2\}$, respectively.
  Note that the transpose matrix of the representation matrix of the boundary operator $d_{p}$ is the representation matrix of its adjoint operator $d_{p}^{\ast}$ with respect to the standard orthogonal basis. However, it is not always true if the chosen basis is not the standard orthogonal basis. For example, let us choose the basis as
  \begin{equation*}
    (1),(2),(3),(4),(5),(12),(25),(13),(35),(14),(45),(125)-(135),(125)-(145).
  \end{equation*}
  Then the corresponding boundary matrix of $d_{2}$ is
  \begin{equation*}
    \tilde{B}_{2}=\left(
                    \begin{array}{cccccc}
                      1 & 1 & -1 & -1 & 0 & 0 \\
                      1 & 1 & 0 & 0 & -1 & -1 \\
                    \end{array}
                  \right).
  \end{equation*}
  If $\tilde{B}_{2}^{T}$ is the representation of $d_{2}^{\ast}$, then we would have
  \begin{equation*}
    \langle d_{2}((125)-(135)),(12)\rangle=1\neq 3=\langle (125)-(135),d_{2}^{\ast}(12)\rangle.
  \end{equation*}
  This is a contradiction. Thus the transpose matrix of the representation matrix of the boundary operator $d_{p}$ is not always the representation matrix of its adjoint operator $d_{p}^{\ast}$ if the chosen basis is not the standard orthogonal basis.
  \end{example}

  \begin{example}\label{example:three_combination}
  Let $\vec{\mathcal{H}}=(V,\vec{E})$ be a hyperdigraph with $V=\{1,2,3\}$. Here, $\vec{E}$ is given by all the sequence of distinct elements in $V$.
  \begin{equation*}
    \vec{E}=\{(1),(2),(3),(12),(13),(21),(23),(31),(32),(123),(132),(213),(231),(312),(321)\}.
  \end{equation*}
  It follows that
  \begin{eqnarray*}
    \Omega_{0}(\vec{\mathcal{H}};\mathbb{R}) &=& \mathrm{span}\{(1),(2),(3)\}, \\
    \Omega_{1}(\vec{\mathcal{H}};\mathbb{R}) &=& \mathrm{span}\{(12),(13),(21),(23),(31),(32)\}, \\
    \Omega_{2}(\vec{\mathcal{H}};\mathbb{R}) &=& \mathrm{span}\{(123),(132),(213),(231),(312),(321)\}.
  \end{eqnarray*}
  We choose the standard orthogonal basis as
  \begin{equation*}
    (1),(2),(3),(12),(13),(21),(23),(31),(32),(123),(132),(213),(231),(312),(321).
  \end{equation*}
  Then the calculation is shown in Table \ref{table:three_combination}.
  \begin{table}[h]
  \caption{Illustration of hyperdigraph in Example \ref{example:three_combination}}\label{table:three_combination}
  \centering
  \begin{small}
  \begin{tabular}{c|c|c|c}
    \hline
    $n$ & $n=0$ & $n=1$ & $n=2$ \\
    \hline
  %  $\Omega_{n}(\vec{\mathcal{H}};\mathbb{R})$ & $\mathrm{span}\{(1),(2),(3)\}$ &  $\mathrm{span}\{(12),(13),(21),(23),(31),(32)\}$ & $\mathrm{span}\{(123),(132),(213),(231),(312),(321)\}$ \\
  %  \hline
    $B_{n+1}$ & $\left(\begin{array}{ccc}
                    -1 & 1 & 0 \\
                    -1 & 0 & 1 \\
                    1 & -1 & 0 \\
                    0 & -1 & 1 \\
                    1 & 0 & -1 \\
                    0 & 1 & -1 \\
                    \end{array}
                     \right)$ & $\left(
                                            \begin{array}{cccccc}
                                              1 & -1 & 0 & 1 & 0 & 0 \\
                                              -1 & 1 & 0 & 0 & 0 & 1 \\
                                              0 & 1 & 1 & -1 & 0 & 0 \\
                                              0 & 0 & -1 & 1 & 1 & 0 \\
                                              1 & 0 & 0 & 0 & 1 & -1 \\
                                              0 & 0 & 1 & 0 & -1 & 1 \\
                                            \end{array}
                                          \right)$
                     &  $6\times 0$ empty matrix \\
    \hline
    $L_{n}$ & $\left(
                \begin{array}{ccc}
                  4 & -2 & -2 \\
                  -2 & 4 & -2 \\
                  -2 & -2 & 4 \\
                \end{array}
              \right)$
     & $\left(
          \begin{array}{cccccc}
          5&	-1&	-2&	0&	0&	-1\\
          -1&	5&	0&	-1&	-2&	0\\
          -2&	0&	5&	-1&	-1&	0\\
          0&	-1&	-1&	5&	0&	-2\\
          0&	-2&	-1&	0&	5&	-1\\
          -1&	0&	0&	-2&	-1&	5\\
          \end{array}
        \right)$
      & $\left(
          \begin{array}{cccccc}
          3&	-2&	-2&	1&	1&	0\\
          -2&	3&	1&	0&	-2&	1\\
          -2&	1&	3&	-2&	0&	1\\
          1&	0&	-2&	3&	1&	-2\\
          1&	-2&	0&	1&	3&	-2\\
          0&	1&	1&	-2&	-2&	3\\
          \end{array}
        \right)$
       \\
         \hline
    $\beta_{n}$ & 1 & 0 & 2 \\
      \hline
    $\mathbf{Spec}(L_{n})$ & \{0,6,6\} & \{1,4,4,6,6,9\} & \{0,0,1,4,4,9\} \\
    \hline
  \end{tabular}
  \end{small}
  \end{table}
%  \[\bordermatrix{%
%                     &(1)&(2)&(3)\cr
%                    (12) &-1 & 1 & 0 \cr
%                    (13) &-1 & 0 & 1 \cr
%                    (21) &1 & -1 & 0 \cr
%                    (23) &0 & -1 & 1 \cr
%                    (31) &1 & 0 & -1 \cr
%                    (32) &0 & 1 & -1 \cr}
%                     \]
%  $\begin{array}{@{}r@{}c@{}c@{}c@{}l@{}}
%            & (1) & {\quad }(2) & {\quad }(3)  \\
%           \left.\begin{array}{c}
%            (12) \\
%            (13) \\
%            (21) \\
%            (23) \\
%            (31) \\
%            (32) \\
%            \end{array}\right(
%            & \begin{array}{c} -1 \\  1  \\ -1  \\  0 \\ 0 \\ 0  \end{array}
%            & \begin{array}{c} -1 \\  1  \\  0  \\ -1 \\ 0 \\ 0  \end{array}
%            & \begin{array}{c}  0 \\  0  \\ -1  \\  0 \\-1 \\ 1  \end{array}
%            & \left)\begin{array}{c} \\ \\  \\ \\ \\ \\
%            \end{array}\right.\end{array}$
  \end{example}

%  \begin{example}
%  Let $\vec{\mathcal{H}}=(V,\vec{E})$ be a hyperdigraph with $V=\{1,2,\dots,n\}$. Here, $\vec{E}$ is given by all the sequence of distinct elements in $V$.
%  \end{example}

  \begin{example}\label{example:example4}
  Let $\vec{\mathcal{H}}=(V,\vec{E})$ be a hyperdigraph with $V=\{1,2,3,4,5\}$. Here, $\vec{E}$ is the set of $(p,1)$-shuffles as follows
  \begin{equation*}
    \vec{E}=\{(1)(2),(2)(1),(1)(3),(3)(2),(4)(2),(4)(3),(5)(2),(14)(2),(14)(3),(15)(2),(34)(2)\}.
  \end{equation*}
  By Eq. (\ref{equation:omega_chaincomplex}), we obtain that
  \begin{eqnarray*}
  % \nonumber to remove numbering (before each equation) \\
    \Omega_{0}(\vec{\mathcal{H}};\mathbb{R})  &=&0,\\
    \Omega_{1}(\vec{\mathcal{H}};\mathbb{R}) &=& \mathrm{span}\{(1)(2)+(2)(1),(1)(3)+(3)(2)+(2)(1),(3)(2)-(4)(2)+(4)(3)\}, \\
    \Omega_{2}(\vec{\mathcal{H}};\mathbb{R}) &=& \mathrm{span}\{(14)(2)-(14)(3)\}.
  \end{eqnarray*}
  We choose the standard orthogonal basis given by
  \begin{equation*}
  \begin{split}
     & \frac{1}{\sqrt{2}}(1)(2)+\frac{1}{\sqrt{2}}(2)(1),-\frac{1}{\sqrt{10}}(1)(2)+\frac{1}{\sqrt{10}}(2)(1)+\frac{2}{\sqrt{10}}(1)(3)+\frac{2}{\sqrt{10}}(3)(2), \\
      & \frac{1}{\sqrt{65}}(1)(2)-\frac{1}{\sqrt{65}}(2)(1)-\frac{2}{\sqrt{65}}(1)(3)+\frac{3}{\sqrt{65}}(3)(2)-\frac{5}{\sqrt{65}}(4)(2)+\frac{5}{\sqrt{65}}(4)(3),\\
     & \frac{1}{\sqrt{2}}(14)(2)-\frac{1}{\sqrt{2}}(14)(3).
  \end{split}
  \end{equation*}
  Then we have the calculation results shown in Table \ref{table:shuffle}.
\begin{table}
  \centering
  \caption{Illustration of hyperdigraph in Example \ref{example:example4}}\label{table:shuffle}
  \begin{small}
  \begin{tabular}{c|c|c|c}
    \hline
    $n$ & $n=0$ & $n=1$ & $n=2$ \\
    \hline
    $B_{n+1}$ & $0\times 3$ empty matrix & $\left(
                                            \begin{array}{ccc}
                                              -\frac{1}{2} & \frac{3}{2\sqrt{5}} &- \frac{\sqrt{13}}{\sqrt{10}}  \\
                                            \end{array}
                                          \right)$
                     &  $6\times 0$ empty matrix \\
    \hline
    $L_{n}$ & none
     & $\left(
         \begin{array}{ccc}
           \frac{1}{4}& -\frac{3}{4\sqrt{5}} & \frac{\sqrt{13}}{2\sqrt{10}}  \\
           -\frac{3}{4\sqrt{5}}  & \frac{9}{20} & -\frac{3\sqrt{13}}{10\sqrt{2}}  \\
           \frac{\sqrt{13}}{2\sqrt{10}}  & -\frac{3\sqrt{13}}{10\sqrt{2}} & \frac{13}{10} \\
         \end{array}
       \right)$
     & 2 \\
         \hline
    $\beta_{n}$ & 0 & 2 & 0 \\
      \hline
    $\mathbf{Spec}(L_{n})$ & none & \{0,0,2\} & \{2\} \\
    \hline
  \end{tabular}
  \end{small}
  \end{table}

  \end{example}

\section{Persistent   hyperdigraph homology and persistent    hyperdigraph Laplacians }\label{section:persistence_on_hyper}

Topological persistence is a powerful computational tool in the field of topology, which allows us to analyze the topological characteristics of a given dataset. It comprises a set of topological invariants that depend on the filtration parameter of the dataset. By tracking changes in these topological features, one can identify the key parameter values where significant changes occur in the dataset. These values correspond to the topological features that are of interest to us, and they can provide insights into the underlying structure and geometry of the data.

Persistent topological Laplacians, proposed by Wei and coworkers \cite{chen2019evolutionary, wang2019persistent, wei2023topological}, can overcome various limitations of persistent homology or persistent topology. These approaches are more quantitative and specific than persistent homology, and they enable the characterization of non-topological shape evolution and the embedding of physical laws in topological invariants \cite{wei2021persistent}. In this section, we introduce the concepts of persistent   hyperdigraph homology and persistent  hyperdigraph Laplacians. It is worth noting that persistence on topological hyperdigraphs can be reduced to that of topological hypergraphs, since a topological  hypergraph can be seen as a topological  hyperdigraph with trivial directions.

  \subsection{Topological persistence on hyperdigraphs and  hyperdigraph Laplacians}\label{section:persistence}
	
  Let $(\mathbb{R},\leq)$ be a category with objects given by the real numbers and the morphisms given by $a\to b$ for $a\leq b$. A \emph{persistence hyperdigraph} is a functor $\mathcal{P}:(\mathbb{R},\leq)\to \vec{\mathbf{H}}\mathbf{yper}$ from $(\mathbb{R},\leq)$ to the category of hyperdigraphs. For real numbers $a\leq b$, the \emph{$(a,b)$-persistent homology} of  a persistent hyperdigraph homology $\mathcal{P}:(\mathbb{R},\leq)\to \vec{\mathbf{H}}\mathbf{yper}$ is given by
  \begin{equation*}
    H_{p}^{a,b}(\mathcal{P};\mathbb{K})=\im(H_{p}(\mathcal{P}(a);\mathbb{R})\to H_{p}(\mathcal{P}(b);\mathbb{R})).
  \end{equation*}
  The $(a,b)$-\emph{persistent Betti number} of $\mathcal{P}$ is defined to be $\beta_{p}^{a,b}=\dim H_{p}^{a,b}(\mathcal{P};\mathbb{R})$ for $p\geq 0$.

  Next, we formulate persistent  hyperdigraph  Laplacians. Let $\vec{\mathbf{H}}\mathbf{yper}^{\hookrightarrow}$ be the subcategory of $\vec{\mathbf{H}}\mathbf{yper}$ with topological hyperdigraphs as objects and inclusions of hyperdigraphs as morphisms. Let $\mathcal{P}:(\mathbb{R},\leq)\to \vec{\mathbf{H}}\mathbf{yper}^{\hookrightarrow}$ be a persistent hyperdigraph homology.
  For real numbers $a\leq b$, we have an inclusion $j^{a,b}:\mathcal{P}(a)\hookrightarrow \mathcal{P}(b)$ of topological hyperdigraphs. Let $\Omega^{t}_{p}=\Omega^{t}_{p}(\mathcal{P}(a);\mathbb{R})$ for any $t\in \mathbb{R}$. The inclusion $j^{a,b}$ induces an inclusion of chain complexes $\mathfrak{j}^{a,b}_{p}:\Omega^{a}_{p}\hookrightarrow \Omega^{b}_{p}$ for $p\geq 0$.
  Let $\Omega_{p+1}^{a,b}=\{x\in\Omega_{p+1}^{b}|\partial^{b}_{p+1} x\in \Omega_{p}^{a}\}$. Let $d_{p+1}^{a,b}=(\mathfrak{j}_{p}^{a,b})^{\ast}\circ d_{p+1}^{b}\circ\iota^{a,b}_{p+1}$.
  \begin{equation*}
    \xymatrix{
  \Omega_{p+1}^{a}\ar@{->}[rr]^{ d_{p+1}^{a}}\ar@{^{(}->}[dd]_{\mathfrak{j}^{a,b}_{p}}&&\quad \Omega_{p}^{a}\quad\ar@<0.75ex>[rr]^-{\textcolor[rgb]{0.00,0.07,1.00}{ d_{p}^{a}} } \ar@{^{(}->}[dd]^{\mathfrak{j}^{a,b}_{p}}\ar@<0.75ex>[ld]^-{\textcolor[rgb]{0.00,0.07,1.00}{( d_{p+1}^{a,b})^{\ast}}}&&\quad \Omega_{p-1}^{a}\ar@<0.75ex>[ll]^-{\textcolor[rgb]{0.00,0.07,1.00}{( d_{p}^{a})^{\ast}}}\ar@{^{(}->}[dd]^{\mathfrak{j}^{a,b}_{p}}\\
                           &\Omega_{p+1}^{a,b}\ar@<0.75ex>[ru]^-{\textcolor[rgb]{0.00,0.07,1.00}{ d_{p+1}^{a,b}} }\ar@{^{(}->}[ld]_{\iota_{p+1}^{a,b}}&&&                       \\
  \Omega_{p+1}^{b}\ar@{->}[rr]^{ d_{p+1}^{b}}&&\quad \Omega_{p}^{b}\quad\ar@{->}[rr]^{ d_{p}^{b}}&&\quad \Omega_{p-1}^{b}
  }
  \end{equation*}
  The \emph{$p$-th $(a,b)$-persistent hyperdigraph Laplacian} $\Delta_{p}^{a,b}:\Omega_{p}^{a}\to \Omega_{p}^{a}$ is defined by
  \begin{equation*}
    \Delta_{p}^{a,b}= d_{p+1}^{a,b}\circ ( d_{p+1}^{a,b})^{\ast}+( d_{p}^{a})^{\ast}\circ d_{p}^{a}.
  \end{equation*}
  We choose two families of the standard orthogonal bases of $\Omega_{p}^{a} $ and $ \Omega_{p+1}^{a,b}$ for $p\geq 0$, respectively. Let $B_{p}^{a}$ and $B_{p}^{a,b}$ be the representation matrices of $d_{p}^{a}$ and $d_{p}^{a,b}$, respectively. Then, the representation matrix of the Laplacian $\Delta_{p}^{a,b}$ is given by
  \begin{equation*}
    L_{p}^{a,b}= (B_{p+1}^{a,b})^{T}\circ B_{p+1}^{a,b}+B_{p}^{a}\circ (B_{p}^{a})^{T}.
  \end{equation*}
  The spectrum of $L_{p}^{a,b}$ is displayed as
  \begin{equation*}
    \mathbf{Spec}(L_{p}^{a,b})=\{\lambda_{p}^{a,b}(1),\lambda_{p}^{a,b}(2),\dots,\lambda_{p}^{a,b}(n)\},\quad p\geq 0,
  \end{equation*}
  where $n=\dim \Omega_{p}^{a}$. The smallest non-zero eigenvalue in $\mathbf{Spec}(L_{p}^{a,b})$ is denoted by $\tilde{\lambda}_{p}^{a,b}$, which has relationship with the Cheeger isoperimetric constant.

  By the persistent Hodge decomposition theorem \cite{liu2023algebraic}, one has
  \begin{theorem}
  $\Omega_{p}^{a}=\ker \Delta_{p}^{a,b}\oplus \im d_{p+1}^{a,b}\oplus \im ( d_{p}^{a})^{\ast}$, where $\ker \Delta_{p}^{a,b}\cong H_{p}^{a,b}(\mathcal{P};\mathbb{R})$.
  \end{theorem}
  Given a persistent hyperdigraph $\mathcal{P}:(\mathbb{R},\leq)\to \vec{\mathbf{H}}\mathbf{yper}$, the theorem says that the number of zero eigenvalues of $\Delta_{p}^{a,b}$ equals to the $(a,b)$-persistent Betti number $\beta_{p}^{a,b}$.
  \begin{corollary}
  $N(\Delta_{p}^{a,b})=\beta_{p}^{a,b}$, where $N(\Delta_{p}^{a,b})$ denotes the number of zeros of $\Delta_{p}^{a,b}$.
  \end{corollary}
  If $a=b$, we have a decomposition
  \begin{equation}\label{equation:decom}
  \Omega_{p}^{a}=\ker \Delta_{p}^{a}\oplus \im d_{p+1}^{a}\oplus \im ( d_{p}^{a})^{\ast},
  \end{equation}
  where $\ker \Delta_{p}^{a}\cong H_{p}^{a}(\mathcal{P};\mathbb{R})$. Let $\pi_{p}^{a}:\Omega_{p}^{a}\to \ker \Delta_{p}^{a}$. Then for any element $x\in \Omega_{p}^{a}$, the element $\psi_{p}^{a}(x)$ is the harmonic component of $x$. Let $\mathcal{H}_{p}^{a}=\ker\Delta_{p}^{a}$. For real numbers $a\leq b$, we denote the \emph{$(a,b)$-persistent harmonic space} by
  \begin{equation*}
  \mathcal{H}^{a,b}_{p}=\im (\pi_{p}^{b}\circ\mathfrak{j}^{a,b}_{p}:\mathcal{H}^{a}_{p}\to \mathcal{H}^{b}_{p}).
  \end{equation*}
  \begin{theorem}\label{theorem:isomorphism}
  For any real numbers $a\leq b$, we have
  \begin{equation*}
  \mathcal{H}_{p}^{a,b}\cong H_{p}^{a,b}(\mathcal{P};\mathbb{R}),\quad p\geq 0.
  \end{equation*}
  \end{theorem}
  \begin{proof}
  We will first prove that $i^{a}:\mathcal{H}_{\ast}^{a}\hookrightarrow \Omega_{\ast}^{a}$ is a quasi-isomorphism. Indeed, for any $x\in \mathcal{H}_{\ast}^{a}$, if $i^{a}(x)$ is a boundary in $\Omega_{\ast}^{a}$, then $x$ is a boundary in $\Omega_{\ast}^{a}$. By Eq. (\ref{equation:decom}), one has $x=0$. So $H(i^{a}):\mathcal{H}_{\ast}^{a}\to H^{a}_{\ast}(\mathcal{P};\mathbb{R})$ is an injection. On the other, for any cycle $z$ in $\Omega_{\ast}^{a}$, by Eq. 
	(\ref{equation:decom}), we can write
\begin{equation*}
  z=z_{0}+d^{a}_{\ast}z_{1}
\end{equation*}
for some $z_{0}\in \mathcal{H}_{\ast}^{a}$ and $z_{1}\in \im d_{\ast}^{a}$. Thus we have $H(i^{a})[z_{0}]=[z_{0}]=[z]$. So the map $H(i^{a}):\mathcal{H}_{\ast}^{a}\to H^{a}_{\ast}(\mathcal{P};\mathbb{R})$ is a surjection.

Now, to prove our theorem, it suffices to show the following diagram commutates.
 \begin{equation*}
  \xymatrix{
  \mathcal{H}^{a}_{p}\ar@{->}[rr]^{\pi_{p}^{b}\circ\mathfrak{j}^{a,b}_{p}}\ar@{->}[d]_{\cong}&& \mathcal{H}^{b}_{p}\ar@{->}[d]^{\cong}\\
   H^{a}_{p}(\mathcal{P};\mathbb{R})\ar@{->}[rr]^{H_{p}(\mathfrak{j}^{a,b}_{\ast})}&& H^{b}_{p}(\mathcal{P};\mathbb{R}) 
  }
 \end{equation*}
For any $x\in \mathcal{H}^{a}_{p}$, we have
\begin{equation*}
  H_{p}(i^{b})\circ\pi_{p}^{b}\circ\mathfrak{j}^{a,b}_{p}(x)=[\pi_{p}^{b}(x)].
\end{equation*}
On the other hand, we obatin
\begin{equation*}
  H_{p}(\mathfrak{j}^{a,b}_{\ast})\circ H_{p}(i^{a})(x)=H_{p}(\mathfrak{j}^{a,b}_{\ast}\circ i^{a})(x)=[x].
\end{equation*}
By Eq. (\ref{equation:decom}), we can write
\begin{equation*}
  z=\pi_{p}^{b}(x)+d^{a}_{\ast}z_{1}
\end{equation*}
for some $z_{1}\in \im d_{\ast}^{a}$. It follows that $H_{p}(i^{b})\circ\pi_{p}^{b}\circ\mathfrak{j}^{a,b}_{p}(x)=H_{p}(\mathfrak{j}^{a,b}_{\ast})\circ H_{p}(i^{a})(x)$. Thus the above diagram is commutative. The desired result follows.
  \end{proof}
Theorem \ref{theorem:isomorphism} says that the $(a,b)$-persistent homology coincides with the $(a,b)$-persistent harmonic space. Or more precisely, the homology and the harmonic space possess the same persistence.

At the end of this section, it is worth noting that a hypergraph homology can be regarded as a hyperdigraph homology with trivial directions, and therefore, the definitions and results discussed above are also applicable to topological  hypergraphs. Consequently, we do not provide a redundant description of persistent hypergraph Laplacians in this context.

  \subsection{Volume-based filtration}
In this section, our focus is on constructing the persistence of point sets using the topological hyperdigraph model, which involves considering the volume of hyperedges. To illustrate this approach, we provide a specific example construction.

  Let $\vec{\mathcal{H}}=(V,\vec{E})$ be a hyperidgraph. Let $f$ be a real-valued function on $\vec{E}$. For each real number $a$, we have a hyperdigraph $\vec{\mathcal{H}}^{f}(a)=(V,\vec{E}(a))$, where $\vec{E}(a)=\{x\in E|f(x)\leq a\}$. Thus $\vec{\mathcal{H}}^{f}:(\mathbb{R},\leq)\to \mathbf{Hyper},a\mapsto \vec{\mathcal{H}}^{f}(a)$ is a persistent hyperdigraph homology. Then we have the persistent  hyperdigraph homology of function $f$ on $\vec{\mathcal{H}}$ given by
  \begin{equation*}
    H_{p}^{a,b}(\vec{\mathcal{H}}^{f};\mathbb{R})=\im(H_{p}(\vec{\mathcal{H}}^{f}(a);\mathbb{R})\to H_{p}(\vec{\mathcal{H}}^{f}(b);\mathbb{R})).
  \end{equation*}
  The corresponding persistent hyperdigraph Laplacian for $f$ on $\mathcal{H}$ can be built in a similar way.

  Now, let $\vec{\mathcal{H}}=(V,\vec{E})$ be a hyperidgraph homology such that the vertex set $V$ is set of points embedded in the Euclidean space $\mathbb{R}^{n}$.
  For each directed hyperedge $e=(v_{0},v_{1},\dots,v_{p})$, let $A=(v_{1}-v_{0},v_{2}-v_{0},\dots,v_{p}-v_{0})$. We have a volume given by
  \begin{equation*}
    \vol(e)=\frac{1}{p!}\sqrt{|\det(A^{T}A)|},\quad p\geq 1.
  \end{equation*}
  We make the convention that $\vol (\{v_{0}\})=0$ for any $v_{0}\in V$.
  Then $f(e)=\left(p!\vol(e)\right)^{\frac{1}{p}}=|\det(A^{T}A)|^{\frac{1}{2p}}$ is a real-valued function defined on the set $\mathbf{P}(V)$. Here, we set $f(\{v_{0}\})=0$ for any $v_{0}\in V$. The volume-based $(a,b)$-persistent  hyperdigraph  homology is given by
  \begin{equation*}
    H_{p}^{a,b}(\vec{\mathcal{H}}^{f};\mathbb{R})=\im(H_{p}(\vec{\mathcal{H}}^{f}(a);\mathbb{R})\to H_{p}(\vec{\mathcal{H}}^{f}(b);\mathbb{R})).
  \end{equation*}
  The  $p$-th $(a,b)$-persistent hyperdigraph Laplacian $\Delta_{p}^{a,b}:\Omega_{p}(\vec{\mathcal{H}}^{f}(a);\mathbb{R})\to \Omega_{p}(\vec{\mathcal{H}}^{f}(a);\mathbb{R})$ for $\vec{\mathcal{H}}^{f}$ is given by
  \begin{equation*}
    \Delta_{p}^{a,b}= d_{p+1}^{a,b}\circ ( d_{p+1}^{a,b})^{\ast}+( d_{p}^{a})^{\ast}\circ d_{p}^{a}.
  \end{equation*}
  Here, $d_{p}^{a}$ and $d_{p}^{a,b}$ are defined as in Section \ref{section:persistence}.

  \begin{example}\label{example:volume}
  Consider hyperidgraph $\vec{\mathcal{H}}=(V,\vec{E})$ such that $V=\{v_{0},v_{1},v_{2}\}$ with $v_{0}=(0,0),v_{1}=(1,2),v_{2}=(2,1)$ in the Euclidean space $\mathbb{R}^{2}$ and
  \begin{equation*}
    \vec{E}=\{(v_{0}),(v_{1}),(v_{2}),(v_{0}v_{1}),(v_{1}v_{2}),(v_{0}v_{2}),(v_{0}v_{1}v_{2})\}.
  \end{equation*}
  It follows that
  \begin{equation*}
    f((v_{0},v_{1}))=\sqrt{5},\quad f((v_{0},v_{2}))=\sqrt{5},\quad f((v_{1},v_{2}))=\sqrt{2},\quad f((v_{0},v_{1},v_{2}))=\sqrt{3}.
  \end{equation*}
  Then we have a filtration of hypergraphs given by
  \begin{eqnarray*}
    \vec{\mathcal{H}}^{f}(0) &=& (V,\{(v_{0}),(v_{1}),(v_{2})\}), \\
    \vec{\mathcal{H}}^{f}(\sqrt{2}) &=& (V,\{(v_{0}),(v_{1}),(v_{2}),(v_{1}v_{2})\}), \\
    \vec{\mathcal{H}}^{f}(\sqrt{3}) &=& (V,\{(v_{0}),(v_{1}),(v_{2}),(v_{1}v_{2}),(v_{0}v_{1}v_{2})\}), \\
    \vec{\mathcal{H}}^{f}(\sqrt{5}) &=& (V,\{(v_{0}),(v_{1}),(v_{2}),(v_{0}v_{1}),(v_{1}v_{2}),(v_{0}v_{2}),(v_{0}v_{1}v_{2})\}).
  \end{eqnarray*}
  \begin{table}
    \centering
    \caption{The Betti numbers and spectra for the volume-based filtration of hyperdigraphs in Example \ref{example:volume}}\label{table:volume}
    \begin{tabular}{c|c|c|c|c|c|c}
      \hline
      \multirow{2}{0.4in}{$\mathcal{H}^{f}(t)$} &\multicolumn{3}{c|}{Betti numbers}  & \multicolumn{3}{c}{Spectra} \\
       & $\beta_{0}$ & $\beta_{1}$ & $\beta_{2}$ & $p=0$ & $p=1$ & $p=2$ \\
      \hline
      $t=0$ & 3 & 0 & 0 & \{0,0,0\} & none & none  \\
      $t=\sqrt{2}$ & 2 & 0 & 0 & \{0,0,2\} & \{2\} & none  \\
      $t=\sqrt{3}$ & 2 & 0 & 0 & \{0,0,2\} & \{2\} & none  \\
      $t=\sqrt{5}$ & 1 & 0 & 0 & \{0,3,3\} & \{3,3,3\} & \{3\}  \\
      \hline
    \end{tabular}
  \end{table}
  By a straightforward calculation, we obtain the Betti numbers and the spectra for the volume-based filtration of  hyperdigraph homology in Table \ref{table:volume}.
  \end{example}

  \subsection{Distance-based filtration}
  In this section, we will consider the distance-based filtration of a given persistent hyperdigraph Laplacian embedded into the Euclidean space. This construction is similar to the persistence on the Vietoris-Rips complex. The application of this work is mainly based on the distance-based filtration.

  Let $\vec{\mathcal{H}}=(V,\vec{E})$ be a hyperdigraph homology such that $V$ is set of points in the Euclidean space. For each real number $a\in \mathbb{R}$, we set $\mathcal{P}(a)=(V,\vec{E}(a))$, where
  \begin{equation*}
    \vec{E}(a)=\{x\in \vec{E}|\text{ the distance between every pair of points in $x$ is at most $a$}\}.
  \end{equation*}
  Then $\mathcal{P}:(\mathbb{R},\leq)\to \vec{\mathbf{H}}\mathbf{yper}^{\hookrightarrow},a\mapsto \mathcal{P}(a)$ is a persistent hyperdigraph homology. Similarly, the $(a,b)$-persistent hyperdigraph homology of $\vec{\mathcal{H}}$ is given by
  \begin{equation*}
    H_{p}^{a,b}(\vec{\mathcal{H}};\mathbb{R})=\im(H_{p}(\mathcal{P}(a);\mathbb{R})\to H_{p}(\mathcal{P}(b);\mathbb{R})).
  \end{equation*}
  Moreover, we can define the $p$-th $(a,b)$-persistent hyperdigraph Laplacian $\Delta_{p}^{a,b}:\Omega_{p}(\mathcal{P}(a);\mathbb{R})\to \Omega_{p}(\mathcal{P}(a);\mathbb{R})$ by
  \begin{equation*}
    \Delta_{p}^{a,b}= d_{p+1}^{a,b}\circ ( d_{p+1}^{a,b})^{\ast}+( d_{p}^{a})^{\ast}\circ d_{p}^{a}.
  \end{equation*}
  Here, $d_{p}^{a}$ and $d_{p}^{a,b}$ are defined as in Section \ref{section:persistence}. Choose two family of standard orthogonal basis of $\Omega_{p}(\mathcal{P}(a);\mathbb{R})$ and $\Omega_{p+1}^{a,b}$ for $p\geq 0$, respectively. Here, $\Omega_{p+1}^{a,b}=\{x\in\Omega_{p+1}(\mathcal{P}(b);\mathbb{R})|\partial^{b}_{p+1} x\in \Omega_{p}(\mathcal{P}(a);\mathbb{R})\}$. Let $B_{p}^{a}$ and $B_{p}^{a,b}$ be the representation matrices of $d_{p}^{a}$ and $d_{p}^{a,b}$, respectively, with respect to the chosen standard basis. We have the representation matrix of $\Delta^{a,b}_{p}$ given by
  \begin{equation*}
    L_{p}^{a,b}= (B_{p+1}^{a,b})^{T}\circ B_{p+1}^{a,b}+B_{p}^{a}\circ (B_{p}^{a})^{T}.
  \end{equation*}
  The corresponding Betti number $\beta_{p}^{a,b}$ is equal to the number of zero eigenvalues of $L_{p}^{a,b}$. Moreover, the spectra can be arranged in an ascending order as
  \begin{equation*}
    \mathbf{Spec}(L_{p}^{a,b})=\{\lambda_{p}^{a,b}(1),\lambda_{p}^{a,b}(2),\dots,\lambda_{p}^{a,b}(n)\}.
  \end{equation*}
  In particular, the Fiedler value $\lambda_{p}^{a,b}(2)$, the second smallest Laplacian eigenvalue of $L_{p}^{a,b}$, is an important feature in various fields.

  \begin{figure}[!ht]
    \centering
    \includegraphics[width=16cm]{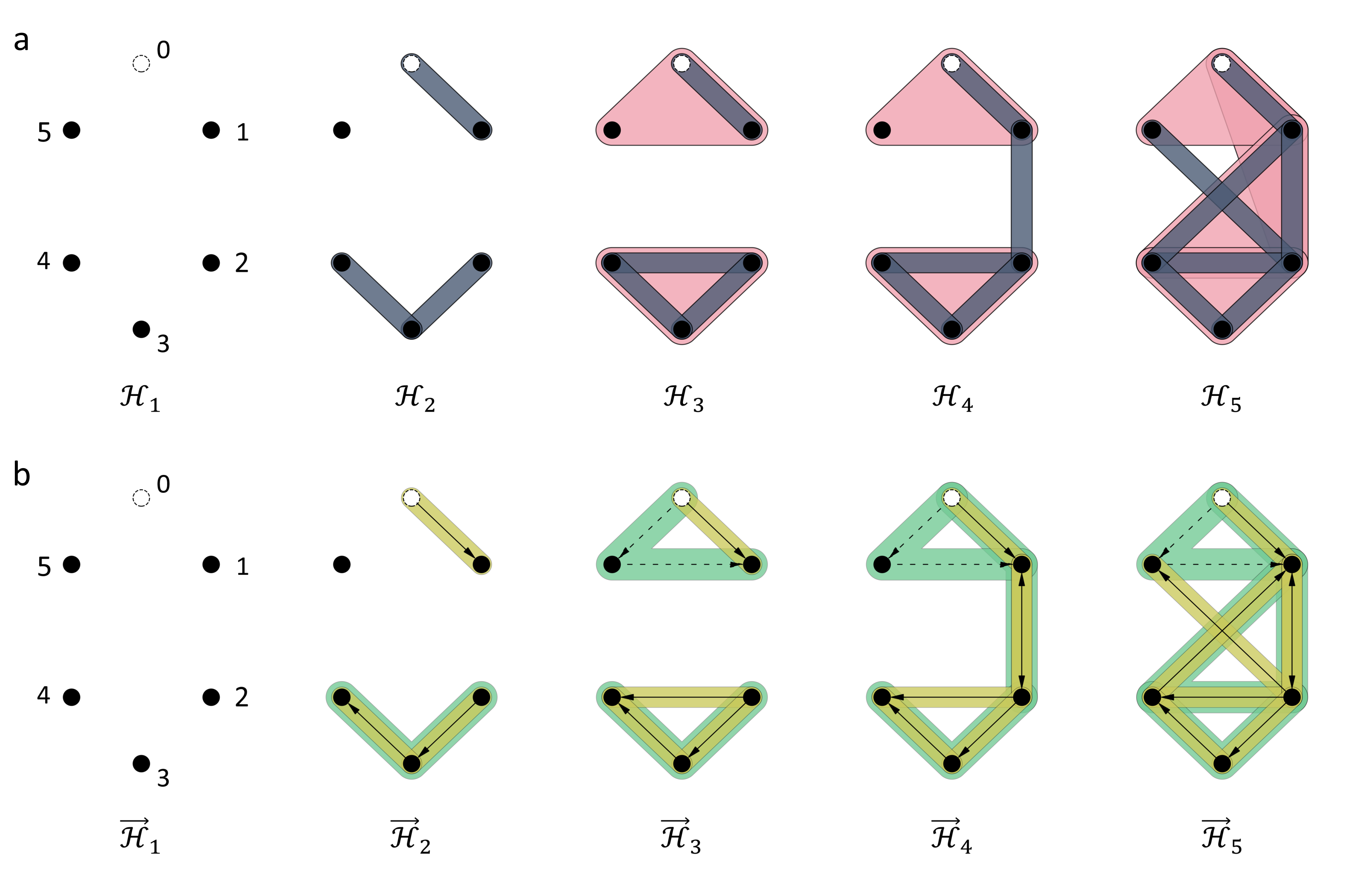}
    % \captionsetup{width=16cm}
    \caption{
      {\bf a} Persistent hypergraph of 6-vertex system ($\vec{\mathcal{H}}$). The distances between the relevant vertices are $D_{45}=D_{12}=6$, $D_{05}=D_{01}=D_{23}=D_{34}=\sqrt{5}$, and $D_{24}=D_{15}=4$.
      {\bf b} Persistent hyperdigraph ($\mathcal{\vec{H}}$) of 6-vertex system. The distances between the corresponding vertices are the same as {\bf a}.
    }
    \label{figure:persistence_example}
  \end{figure}

  As shown in Figure \ref{figure:persistence_example}, the hypergraph homology ($\mathcal{H}$) and hyperdigraph homology ($\mathcal{\vec{H}}$) are defined on the same vertex set $V=\{0, 1, 2, 3, 4, 5\}$ in the Euclidean space. The distance between $i$th and $j$th vertices are denoted as $D_{ij}$. Here, we have $D_{45}=D_{12}=6$, $D_{05}=D_{01}=D_{23}=D_{34}=\sqrt{5}$, and $D_{24}=D_{15}=4$. The largest hypergraph homology ($\mathcal{H}_5$) and largest hyperdigraph homology ($\mathcal{\vec{H}}_5$) are predefined in this example. The hyperedge set of $\mathcal{H}_5$ is the same as the Example \ref{example:hypergraph2}. The $\mathcal{H}_1, \mathcal{H}_2,\dots,\mathcal{H}_5$ in Figure \ref{figure:persistence_example}{\bf a} represent the hypergraph homology along the distance-based filtration, where the corresponding filtration parameters are 0, $\sqrt{5}$, 4, 6, and $\sqrt{53}$, respectively. The $\mathcal{\vec{H}}_1,\mathcal{\vec{H}}_2, \dots , \mathcal{\vec{H}}_5$ in Figure \ref{figure:persistence_example}{\bf b} represent the hyperdigraph homology along the distance-based filtration. And the filtration parameters for $\mathcal{\vec{H}}_n$ ($n=1,2,3,4,5$) are the same as the corresponding filtration parameters for $\mathcal{H}_n$.

  \begin{figure}[!ht]
    \centering
    \includegraphics[width=16cm]{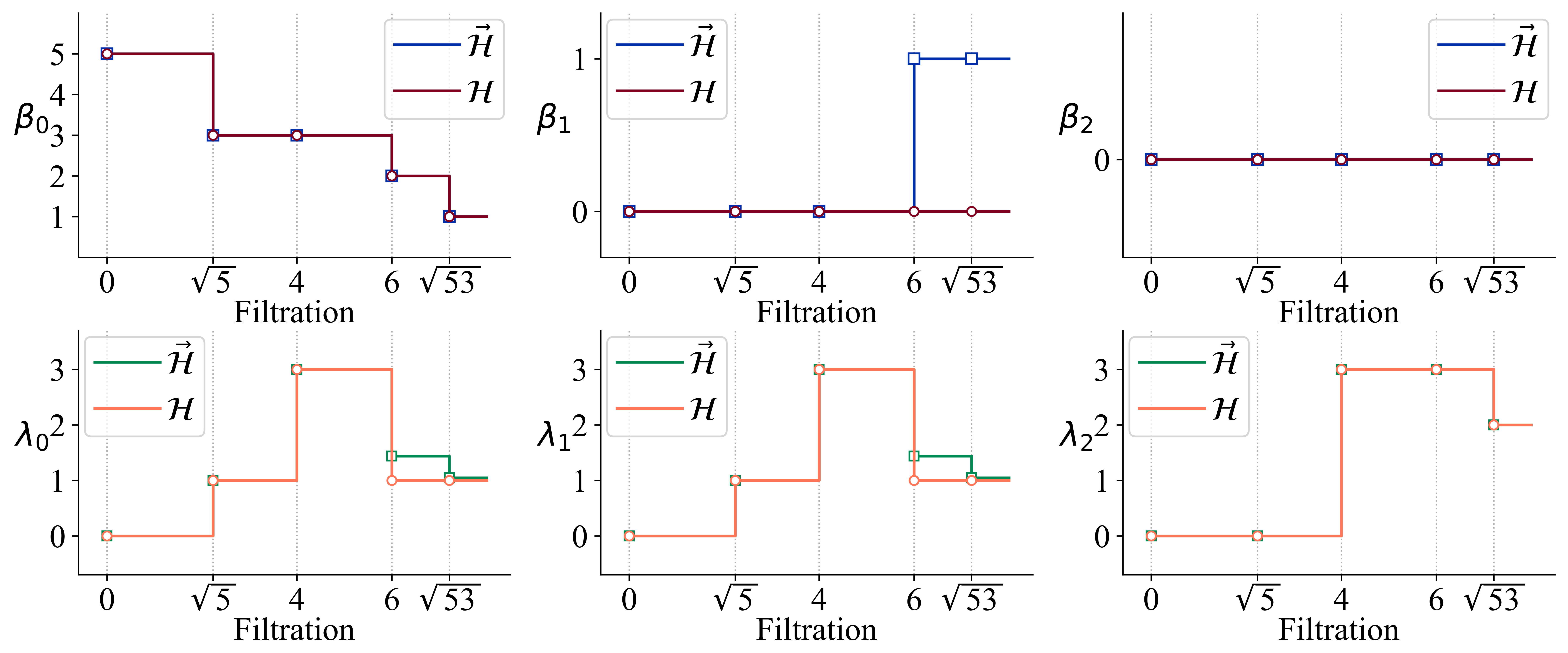}
    % \captionsetup{width=16cm}
    \caption{
      Comparison of persistent  hypergraph  Laplacians and persistent  hyperdigraph  Laplacians. The hollow square markers ($\mathcal{\vec{H}}$) and circle markers ($\mathcal{H}$) in each subfigure correspond to the five filtration stages in Figure \ref{figure:persistence_example}, i.e., $\mathcal{\vec{H}}_1$, $\mathcal{\vec{H}}_2$, ... , $\mathcal{\vec{H}}_5$. and $\mathcal{H}_1$, $\mathcal{H}_2$, ... , $\mathcal{H}_5$.
    }
    \label{figure:persistence_betti_lambda}
  \end{figure}

   Figure \ref{figure:persistence_betti_lambda} illustrates   persistent  hypergraph  Laplacians  and   persistent   hyperdigraph Laplacians for the objects given in Figure \ref{figure:persistence_example}{\bf a} and Figure \ref{figure:persistence_example}{\bf b}. For the 6-vertex system in this example, $\beta_0$ and $\beta_2$ are always the same throughout the filtration of the  persistent hypergraph homology and  persistent hyperdigraph homology, with $\beta_0$ decreasing as the filtration parameter increases and $\beta_2$ always being 0. The $\beta_1$ of the  persistent hypergraph homology is always 0 throughout, while when the filtration parameter starts at 6, the $\beta_1$ of the   persistent hyperdigraph homology changes from 0 to 6 and persists to the end, which means that there is a 1-dimensional cycle formation on the corresponding  persistent hyperdigraph homology. For the non-harmonic spectra of two types of  persistent topological Laplacians, as shown in the bottom panel of Figure \ref{figure:persistence_betti_lambda}. Except for $\lambda_2$, both $\lambda_0$ and $\lambda_1$ will produce different eigenvalues when the parameters start from 6, and the corresponding  persistent hypergraph  Laplacian  and  persistent hyperdigraph  Laplacians are $\mathcal{H}_4$ and $\mathcal{\vec{H}}_4$, respectively. For the   persistent hypergraph  Laplacian, when the filtration parameter is 6, 1-hyperedge $\{1, 2\}$ are newly generated, and for the  persistent hyperdigraph  Laplacian, directed 1-hyperedges $(12)$ and $(21)$ and directed 2-hyperedge $(012)$ are generated. It is worth noting that for the  persistent hyperdigraph  Laplacian,   $(12)$ and $(21)$ are two different directed 1-hyperedges, while the persistent hypergraph  Laplacian has only 1-hyperedge $\{1, 2\}$, which indicates that the  persistent hyperdigraph  Laplacian can distinguish vertex   directions and contain richer information.

  \begin{figure}[!ht]
    \centering
    \includegraphics[width=16cm]{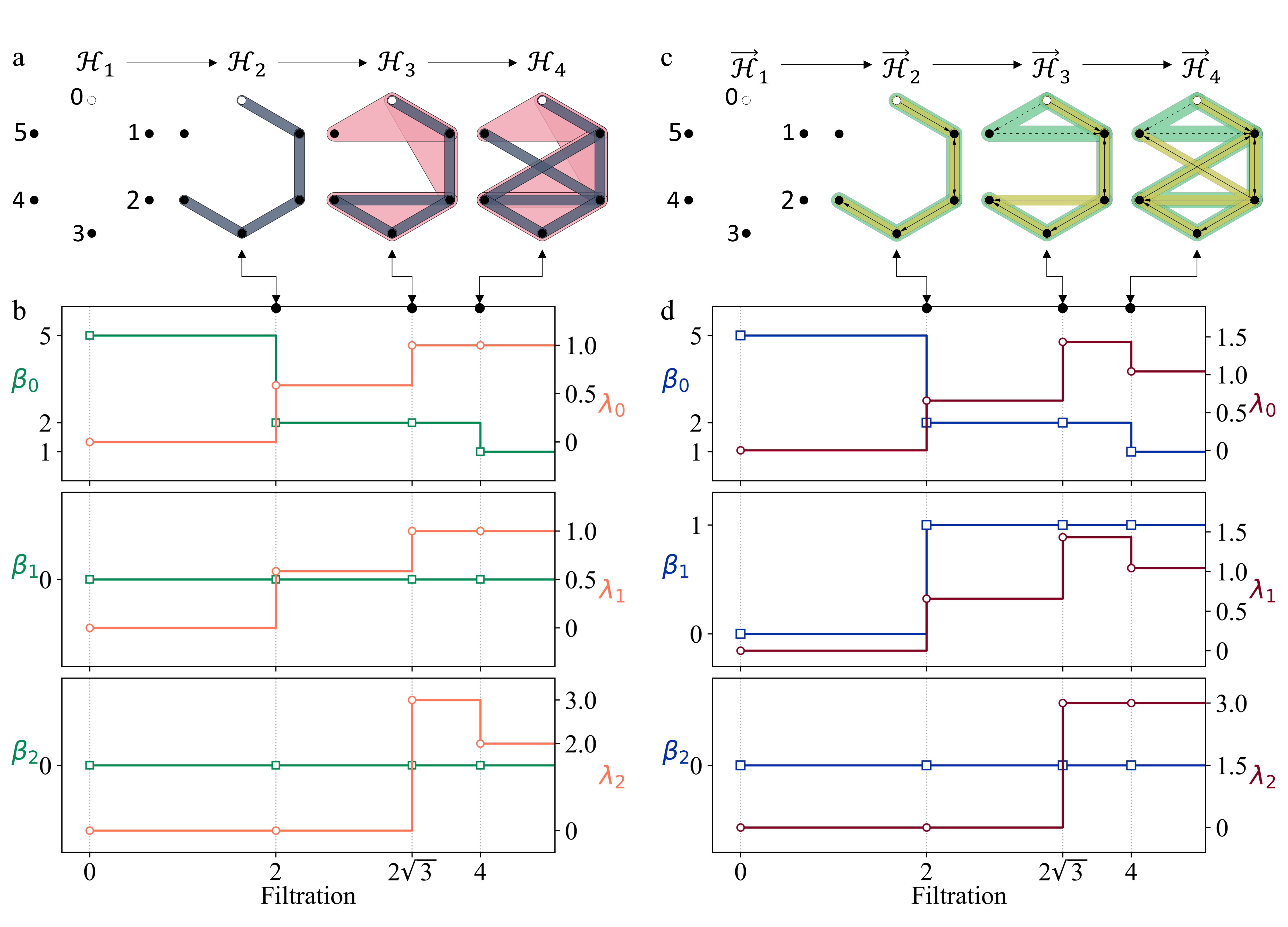}
    % \captionsetup{width=16cm}
    \caption{
      {\bf a} Persistent hypergraph Laplacians of a 6-vertex system ($\mathcal{H}$). All vertices come from the vertices of a regular hexagon with sides of length 2, i.e., $D_{01}=D_{12}=D_{23}=D_{34}=D_{45}=D_{50}=2$.
      {\bf b} Comparison of Betti numbers ($\beta_{0}$, $\beta_{1}$, and $\beta_{2}$) and the smallest eigenvalue of the non-harmonic spectra ($\lambda_{0}$, $\lambda_{1}$, and $\lambda_{2}$) of persistent hypergraph  Laplacians.
      {\bf c} Persistent hyperdigraph  Laplacians of a 6-vertex system ($\vec{\mathcal{H}}$). The distances between the corresponding vertices are the same as {\bf a}.
      {\bf d} Comparison of Betti numbers ($\beta_{0}$, $\beta_{1}$, and $\beta_{2}$) and the smallest eigenvalue of the non-harmonic spectra ($\lambda_{0}$, $\lambda_{1}$, and $\lambda_{2}$) of persistent hyperdigraph  Laplacians. The corresponding Betti numbers can not tell the changes from $\mathcal{H}_3$ to $\mathcal{H}_4$ and $\vec{\mathcal{H}}_3$ to $\vec{\mathcal{H}}_4$, while the $\lambda_{0}$, $\lambda_{1}$, and $\lambda_{2}$ can capture the difference.
    }
    \label{figure:persistence_both_compare}
  \end{figure}

To demonstrate that persistent topological Laplacians provide more information than persistent topological homologies, we present another example in Figure \ref{figure:persistence_both_compare}{\bf a} and Figure \ref{figure:persistence_both_compare}{\bf c}. We use a system of 6 vertices with positive hexagons, where the side lengths are $D_{01}=D_{12}=D_{23}=D_{34}=D_{45}=D_{50}=2$. The largest hypergraph homology ($\mathcal{H}_4$) and hyperdigraph homology ($\vec{\mathcal{H}}_4$) of the system are consistent with the ones in Figure \ref{figure:persistence_example}. The persistent hypergraph Laplacian features of dimensions 0, 1, and 2 are shown in Figure \ref{figure:persistence_both_compare}{\bf b}, where the green line shows the Betti numbers of the corresponding filtration parameter, and the orange color shows the smallest eigenvalue of the non-harmonic spectra for the corresponding filtration parameter. We observe that when the filtration parameter is $2\sqrt{2}$, the corresponding Betti numbers are unable to distinguish the changes from $\mathcal{H}_2$ to $\mathcal{H}_3$, while the $\lambda_{0}$, $\lambda_{1}$, and $\lambda_{2}$ can identify the differences. This suggests that using only the harmonic spectra, i.e., the Betti numbers, can sometimes overlook changes within the structure during the filtration process, while the non-harmonic spectra can capture the variations more sensitively. We also find a similar phenomenon for persistent hyperdigraph Laplacian features in Figure \ref{figure:persistence_both_compare}{\bf d}. When the filtration parameter increases to $2\sqrt{2}$, the directed edges $(5,1)$ and $(2,4)$ are formed. However, the corresponding Betti numbers cannot reflect the changes in connectivity, while   $\lambda_{0}$, $\lambda_{1}$, and $\lambda_{2}$ can detect the variations.

\section{Application}\label{section:application}

In this section, we demonstrate the application of persistent hyperdigraph Laplacians to analyze a protein-ligand system. Specifically, we use the protein-ligand complex structure of 1a99 from the Protein Data Bank (PDB) as a case study. As depicted in Figure \ref{figure:persistent_dihypergraph}{\bf b}, the complex includes a small ligand, C$4$H${14}$N$_2^{2+}$. To facilitate visualization of the protein-ligand binding, we limit our analysis to the C atoms from the protein located within 4 {\AA} of the ligand, as illustrated in Figure \ref{figure:persistent_dihypergraph}{\bf a}.

  \begin{figure}[!ht]
    \centering
    \includegraphics[width=16cm]{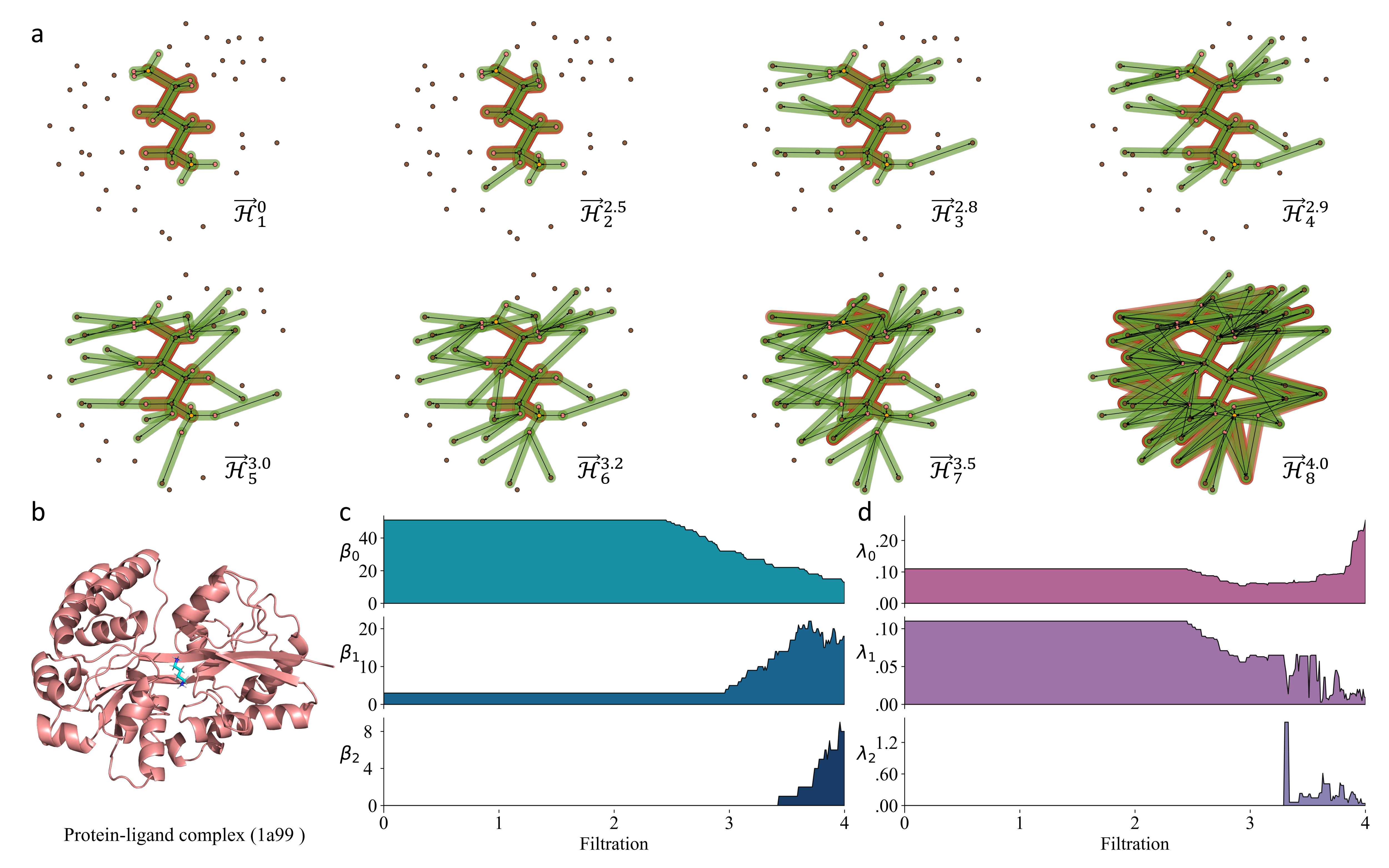}
    % \captionsetup{width=16cm}
    \caption{Illustration of the persistent hyperdigraph Laplacian analysis of a protein-ligand complex (PDB ID: 1a99).
        {\bf a} Illustration of a filtration-induced topological hyperdigraph. Only C-atoms within 4 {\AA} from the ligand are considered for the protein structure (brown dots). The green line segment indicates the directed 1-hyperedges and the orange line indicates the directed 2-hyperedges.
        {\bf b} The three dimensional structure of the protein-ligand complex.
        {\bf c} Betti numbers $\beta_n$ ($n$=0,1,2) and the smallest eigenvalues of non-harmonic spectra $\lambda_n$ ($n$=0,1,2) of persistent hyperdigraph Laplacian for the protein-ligand complex.
    }
    \label{figure:persistent_dihypergraph}
  \end{figure}

Figure \ref{figure:persistent_dihypergraph} depicts the analysis of a protein-ligand complex (PDB ID: 1a99) using persistent hyperdigraph Laplacians. For this analysis, we focus only on the interactions between the ligand and the surrounding carbon atoms from the protein, represented by directed edges. The direction of each edge is determined by the electronegativity of the atoms involved. Specifically, an atom with lower electronegativity points to an atom with higher electronegativity, while two atoms with the same electronegativity are connected by two edges with opposite directions. For the protein-ligand complex, we can determine the following atomic directions: H(2.2) $\rightarrow$ S(2.44) $\rightarrow$ C(2.5) $\rightarrow$ N(3.07) $\rightarrow$ O(3.5), where the values in parentheses denote the electronegativities of the atoms.

To better capture the interactions between the protein and ligand, we exclude the interactions within the protein. However, we include the internal covalent bonding relationships of the ligand, as shown in the first subfigure of Figure \ref{figure:persistent_dihypergraph}{\bf a}. The persistent hyperdigraph Laplacian analysis then reveals important structural changes in the complex, as demonstrated by the persistent Betti numbers and non-harmonic persistent eigenvalues shown in Figure \ref{figure:persistent_dihypergraph}{\bf c} and \ref{figure:persistent_dihypergraph}{\bf d}. These results suggest that non-harmonic persistent spectra are more informative than persistent Betti numbers or harmonic persistent spectra in capturing the complex's structural changes. Therefore, persistent hyperdigraph Laplacians provide a more powerful tool for analyzing   data, particularly when combined with machine learning and deep learning algorithms, as previously discussed in the literature \cite{chen2022persistent}.

Directed 0-hyperedges in the figure are discrete points, while  directed 1-hyperedge are given by the edges with green backgrounds and  directed 2-hyperedges are  labeled by  edges with orange backgrounds. It is worth noting that  directed 1-hyperedges represent two different vertices with  fixed orders and  directed 2-hyperedges represent three different vertices with  fixed orders. Thus, in our  persistent hyperdigraph Laplacian  theory, for a closed loop consisting of two vertices, e.g., 1$\rightarrow$2$\rightarrow$1, will be treated as a directed 1-hyperedge. This is where it differs from persistent  path Laplacian theory \cite{wang2023persistent}, in which a closed loop can generate an arbitrary high-dimensional path. Figure \ref{figure:persistent_dihypergraph}{\bf a} illustrates the filtration-induced persistent hyperdigraph $\mathcal{\vec{H}}$, i.e., $\mathcal{\vec{H}}_1^0$, $\mathcal{\vec{H}}_2^{2.5}$, $\mathcal{\vec{H}}_3^{2.8}$, $\mathcal{\vec{H}}_4^{2.9}$, $\mathcal{\vec{H}}_5^{3.0}$, $\mathcal{\vec{H}}_6^{3.2}$, $\mathcal{\vec{H}}_7^{3.5}$, and $\mathcal{\vec{H}}_8^{4.0}$, where the superscripts are the corresponding filtration parameters.

Figures \ref{figure:persistent_dihypergraph}{\bf c} and \ref{figure:persistent_dihypergraph}{\bf d} respectively show the persistent Betti numbers and non-harmonic persistent eigenvalues obtained from persistent hyperdigraph Laplacians. In this analysis, the smallest non-zero eigenvalues of Laplacian matrices were selected to represent the non-harmonic spectral information. The results demonstrate that persistent Betti numbers and non-harmonic persistent eigenvalues exhibit completely different features. For example, at a filtration parameter of about 2.5 (corresponding to the first and second diagrams in Figure \ref{figure:persistent_dihypergraph}{\bf a}), the ligand connects to the C atoms in the protein, forming new directed 1-hyperedges. At this point, the number of directed 0-hyperedges of the complex decreases, and persistent Betti numbers $\beta_0$ decrease. However, the higher dimensional Betti numbers $\beta_1$ and $\beta_2$ remain unchanged. In contrast, $\lambda_0$ and $\lambda_1$ of the non-harmonic persistent spectra decrease significantly, indicating the formation of large connected complexes. Moreover, the range of $\beta_2$ with changes is smaller than that of $\lambda_2$.

These results suggest that non-harmonic persistent spectra are more informative than persistent Betti numbers or harmonic persistent spectra. Therefore, persistent hyperdigraph Laplacians can better capture structural changes than persistent hyperdigraph homology.

%		information obtained in persistent hyperdigraph Laplacians encoded the structural information from different perspectives, i.e., topological and geometric views.

\section{Conclusion}
Since hypergraphs are a generalization of graphs, it is natural to consider hyperdigraphs as a generalization of directed graphs (digraphs). However, a problem arises when trying to embed topological structures into hyperdigraphs.

In this work, we introduce the concept of hyperdigraph homology, or topological hyperdigraphs, to embed topological information into hyperdigraphs. The intrinsic vertex order of directed hyperedges in topological hyperdigraphs allows for the distinction of different sets of vertices, making it possible to distinguish all possible permutations of a given vertex set. As a result, topological hyperdigraphs can be viewed as a generalization of topological hypergraphs. We also develop a method to reduce a hyperdigraph homology to a hypergraph homology.

We introduce a new set of Laplacian methods called topological hyperdigraph Laplacians, which serve as a generalization of hyperdigraph homology. These Laplacians provide both harmonic spectra (related to topological invariants or Betti numbers) and non-harmonic spectra.

Furthermore, we propose persistent hyperdigraph homology and persistent hyperdigraph Laplacians through filtration. The persistent hyperdigraph Laplacians not only return the persistent topological invariants of persistent hyperdigraph homology but also provide non-harmonic persistent spectra to capture the homotopic shape evolution of the data at various scales. We demonstrate the usefulness of these proposed topological methods through numerous examples.

Finally, we explore the application of persistent hyperdigraph Laplacians by characterizing the interactions between a protein and a ligand. We believe that these proposed methods, including hyperdigraph homology, topological hyperdigraph Laplacians, persistent hyperdigraph homology, and persistent hyperdigraph Laplacians, provide a powerful set of new tools for topological data analysis (TDA). We anticipate that combining these methods with machine learning and deep learning algorithms will have a significant impact on data science.

%	the generated features can also be used in the future for large-scale structure prediction and accelerated drug discovery.

\section{Acknowledgments}
This work was supported in part by NIH grants  R01GM126189 and  R01AI164266, NSF grants DMS-2052983,  DMS-1761320, and IIS-1900473,  NASA grant 80NSSC21M0023,  MSU Foundation,  Bristol-Myers Squibb 65109, and Pfizer.
  The work of   Liu  and Wu was supported by Natural Science Foundation of China (NSFC grant no. 11971144), High-level Scientic Research Foundation of Hebei Province, the start-up research fund from BIMSA.

%\bibliographystyle{unsrtnat}  % unsrtnat
%\bibliography{Reference}

\end{document}